\documentclass[en,oneside,bbfont,article]{amsart} 
\usepackage[lmargin = 30mm, rmargin = 30mm, bmargin = 25mm, tmargin=25mm]{geometry}
\usepackage[utf8]{inputenc}
\usepackage{amsmath,amssymb,amsthm}
\usepackage{geometry}
\usepackage{hyperref}
\usepackage{graphicx,float}
\usepackage{tikz}
\usepackage{dsfont,animate}
\usepackage{comment}
\usepackage{tablefootnote}
\usepackage{bbm,bm}
\usepackage{todonotes}
\usepackage{mleftright}

\providecommand{\examplename}{Example}

\newtheorem{theorem}{Theorem}[section]
\newtheorem{proposition}[theorem]{Proposition}

\newtheorem{lemma}[theorem]{Lemma}
\theoremstyle{remark}
\newtheorem{remarkx}[theorem]{Remark}
\newenvironment{remark}
  {\pushQED{\qed}\remarkx}
  {\popQED\endremarkx}
\theoremstyle{definition}
\newtheorem*{example*}{\protect\examplename}

\theoremstyle{plain}

\newtheorem*{assumption*}{Assumption}

\DeclareMathOperator{\sech}{sech}
\newcommand{\D}[1]{\mathop{\mathrm{d}#1}}
\newcommand{\R}{\mathbb{R}}
\newcommand{\conv}{\mathrm{conv}}
\newcommand\E{\mathds{E}}
\newcommand\p{\mathds{P}}

\newcommand{\Var}{\mathrm{Var}}

\newcommand\N{\mathbb{N}}

\newcommand\1{\mathds{1}}

\newcommand{\sfr}{\mathsf{r}}
\newcommand{\sfP}{\mathsf{P}}
\newcommand{\sfA}{\mathsf{A}}
\newcommand{\sfD}{\mathsf{D}}
\newcommand{\sfR}{\mathsf{R}}
\newcommand{\rmH}{\mathcal{H}}

\usepackage{xcolor}

\usepackage{bm}
\usepackage{mathtools}
\mathtoolsset{showonlyrefs}

\usepackage[backend=biber,style=ieee,firstinits=true,citestyle=numeric,sorting=nyt,maxbibnames=99,dashed=false,sortcites=true]{biblatex}
\renewbibmacro{in:}{}
\appto{\bibsetup}{\sloppy}
\usepackage[foot]{amsaddr}

\title[Bounds on the inradius and inverse inradius of the planar Brownian convex hull]{Improved bounds on the inradius and inverse inradius process of the convex hull of planar Brownian motion}
\author{David Kramer-Bang$^{\dag}$ \& Hugo Panzo$^{*}$ \& Stjepan \v{S}ebek$^{\ddag}$}

\address{$^\dag$Department of Mathematics, Aarhus University, DK}
\address{$^*$Department of Mathematics and Statistics, Saint Louis University, USA}
\address{$^\ddag$University of Zagreb Faculty of Electrical Engineering and Computing, HR}

\email{bang@math.au.dk}
\email{hugo.panzo@slu.edu}
\email{stjepan.sebek@fer.unizg.hr}

\begin{document}

\begin{abstract}
        We study the inradius and inverse inradius of the convex hull of standard planar Brownian motion. Our main goals are to improve the existing bounds on the expected values of these quantities and to establish variance bounds. The main contributions of the paper, however, are the following intermediate results we obtain in order to compute our bounds. We find closed-form expressions for the expected values of the minimum and maximum of the ranges of two independent one-dimensional Brownian motions. Besides improving the upper bound on the expected inradius, this result also gives an explicit expression for a quantity that has been investigated repeatedly in the literature and for which only a numerical evaluation has been known so far. Furthermore, we study the exit time of planar Brownian motion from an unbounded region that we call the \emph{hourglass domain} and compute its first two moments. This result leads to a nearly fivefold improvement of the previous upper bound on the expected inverse inradius and plays an essential role in establishing the corresponding upper bound on the variance.
\end{abstract}

\subjclass[2020]{Primary 60D05, 60J65; Secondary 52A10}
\keywords{Brownian motion, convex hull, exit time, inradius.}

\maketitle

\section{Introduction}

Let $\bm W = (\bm W(t) : t\ge 0)$  denote planar Brownian motion. More precisely, $\bm W(t)= (W_1(t), W_2(t))$, where the two coordinates $W_i(t)$, $i=1,2$, are independent, one-dimensional Brownian motions. Typically, $\bm W$ will start at the origin, otherwise the starting point $(x, y)\neq (0, 0)$ will be specified by the path measure $\p_{(x, y)}$ or expectation operator $\E_{(x, y)}$. For a set $\mathcal{A}\subset \R^2$, denote by $\conv \mathcal{A}$ the convex hull spanned by the set $\mathcal{A}$, i.e.,\ the smallest convex subset of $\R^2$ containing $\mathcal{A}$. Let $\rmH(t) = \conv \{\bm W(s): s\in [0,t]\}$ be the convex hull spanned by a path of $\bm W$ run up to time $t$. The convex hull of planar Brownian motion has been investigated for many years. Questions concerning the smoothness properties of its boundary were considered already by P.\ L\'{e}vy in~\cite{Levy}. He conjectured that the boundary of $\rmH(t)$ should be of class $C^1$. This conjecture was proved in~\cite{ElBachir_PhD_thesis}; see also~\cite{Cranston}.

In this paper, we are concerned with questions about a particular geometric functional of the planar Brownian convex hull, namely its inradius, which is the radius of the largest circle contained in $\rmH(t)$. Different geometric functionals of the planar Brownian convex hull have already been studied in the literature. Let $\sfP(t)$ denote the perimeter of $\rmH(t)$. Throughout we (usually) abandon the time index whenever $t=1$, that is we write $\sfP$ for $\sfP(1)$, and similarly for other functionals. The expected value of $\sfP(t)$ was computed in~\cite{Letac}; see also~\cite{Majumdar} and~\cite{MR2551685}. It holds that $\E[\sfP] = \sqrt{8\pi}$. Let $\sfA(t)$ denote the area of $\rmH(t)$. In~\cite{ElBachir_PhD_thesis} it was computed that $\E[\sfA] = \pi/2$. These results were later generalized to the standard planar Brownian bridge~\cite{goldman, Majumdar}, the union of independent standard planar Brownian motions~\cite{Majumdar}, the union of independent standard planar Brownian bridges~\cite{Majumdar}, and various combinations of independent standard planar Brownian motions and bridges~\cite{sebek}. 

When it comes to geometric functionals of the convex hull of standard planar Brownian motion, perimeter and area are the only ones for which explicit expressions for their expectation were computed. Let $\sfD(t)$ denote the diameter of $\rmH(t)$. The authors in~\cite{McRedmond_Xu} found lower and upper bounds for $\E[\sfD]$. Later, the lower bound was improved in~\cite{Jovalekic, Garbit-Raschel}. In~\cite{CPS}, the authors studied the inradius of $\rmH(t)$ as well as its circumradius, which is the radius of the smallest circle that contains $\rmH(t)$.

The authors of~\cite{CPS} also initiated the study of the inverse processes of all the mentioned geometric functionals. Using the notation
\begin{equation}\label{eq:gen_inv_proc}
	\Theta^X(y) = \inf\{t\ge 0 : X(t) > y\}, \quad y \ge 0,
\end{equation}
for the inverse process of a one-dimensional non-decreasing stochastic process $X(t)$, they found upper and lower bounds for the expected value of quantities $\Theta^{\sfP}, \Theta^{\sfA}, \Theta^{\sfD}, \Theta^{\sfR}$ and $\Theta^{\sfr}$, where we again drop the dependence on $y$ in the case $y = 1$.

The novelty of the present paper is twofold. First, we evaluate explicitly the expected value of the minimum of two independent ranges of standard one-dimensional Brownian motions. We do this by first expressing the expected minimum as the integral of the square of Feller’s range-tail expansion from \cite{Feller1951}. The obtained exact expression enables us to improve the upper bound on the expected inradius of the planar Brownian convex hull from \cite{CPS}, and to provide an explicit formula for the expected value of the maximum of two independent ranges of standard one-dimensional Brownian motions which is an object studied in \cite{McRedmond_Xu, Garbit-Raschel, Jovalekic} and for which only numerical approximations are known. We then repeat this calculation for the second moment of the minimum and obtain an explicit formula containing an infinite series. This formula is then used to get an upper bound on the variance of the inradius of the planar Brownian convex hull. Second, we study the exit time of planar Brownian motion from the interior of a carefully calibrated hyperbola, the so-called \emph{hourglass domain}. We find the first and second moments of this exit time, which proves to be a strong tool that allows us to significantly improve the upper bound on the expected inverse inradius of the planar Brownian convex hull from \cite{CPS}, and to provide a variance upper bound for the same object.

\subsection{Main contributions}
Let us denote by $\sfr(t)$ the inradius of $\rmH(t)$. In~\cite{CPS} it was shown that
\begin{equation*}
	0.393 \le \E[\sfr] \le 0.7072.
\end{equation*}
In this paper we improve on the upper bound for $\E[\sfr]$. In order to state this result, we first introduce some notation. Denote by
\begin{equation}\label{eq_riemann_zeta_function_def}
    \zeta(s)=\sum_{n=1}^\infty n^{-s}
\end{equation}
the Riemann zeta function, and by
\begin{equation}\label{eq:defn_beta_diri}
    \beta(s)=\sum_{k=0}^\infty (-1)^k(2k+1)^{-s}
\end{equation} 
the Dirichlet beta function. We are now ready to state our main result related to the improved upper bound for $\E[\sfr]$.
\begin{theorem}\label{thm:ub_exp_inr}Let $\zeta(s)$ be the Riemann zeta function from~\eqref{eq_riemann_zeta_function_def} and $\beta(s)$ the Dirichlet beta function from~\eqref{eq:defn_beta_diri}. Then, for the inradius $\sfr$ of the standard planar Brownian motion $\bm W$, it holds that
\[
\E[\sfr] \le \frac{2(4-\sqrt2)}{\pi^{5/2}}
\zeta \mleft(\frac32\mright)\beta \mleft(\frac32\mright)\le 0.667652.
\]
\end{theorem}
The result from Theorem~\ref{thm:ub_exp_inr} is obtained by using the inequality $\sfr\le \min\{R_1,R_2\}/2$ a.s., where $R_i$ is the range of the $i$-th coordinate $W_i$ (see Lemma~\ref{lem:r_bound_range}). Taking expectations and using the fact that \(R_1\) and \(R_2\) are independent and identically distributed gives
\[
\E[\sfr]\le \frac12  \E[\min\{R_1,R_2\}]
=\frac12\int_0^\infty \p(R\ge x)^2\D x,
\]
where \(R\) is the range of a standard one-dimensional Brownian motion. Feller's expansion~\cite[Eq.~(3.7) \&~(3.8)]{Feller1951} for \(\p(R\ge x)\) then reduces the problem to an explicit one-dimensional integral, and Proposition~\ref{prop:square-tail-integral} of the present paper evaluates this integral in closed form which provides the elegant formula
\begin{equation}\label{eq:explicit_expression_exp_min_range}
    \E[\min\{R_1,R_2\}] = \frac{4(4-\sqrt2)}{\pi^{5/2}}
\zeta \mleft(\frac32\mright)\beta \mleft(\frac32\mright) \approx 1.33530.
\end{equation}
\begin{remark}
    As a consequence of this result, we easily get a closed form expression for the $\E[\max\{R_1, R_2\}]$. In~\cite{McRedmond_Xu} the authors write that it seems hard to explicitly compute $\E[\max\{R_1, R_2\}]$. This expectation then again appears in~\cite{Garbit-Raschel} and~\cite{Jovalekic} as a lower bound for the expected value of the diameter of the planar Brownian convex hull. In both~\cite{Garbit-Raschel} and~\cite{Jovalekic} the $\E[\max\{R_1, R_2\}]$ is evaluated only numerically. We first recall a trivial observation that
    \begin{equation*}
        R_1 + R_2 = \min\{R_1, R_2\} + \max\{R_1, R_2\}.
    \end{equation*}
    Taking expectations of both sides, using~\eqref{eq:explicit_expression_exp_min_range}, and the precise value of the expected range of a standard one-dimensional Brownian motion (see~\cite[Eq.~(1.4)]{Feller1951}), we get
    \begin{equation*}
        \E[\max\{R_1, R_2\}] = 2\sqrt{\frac{8}{\pi}} - \frac{4(4-\sqrt2)}{\pi^{5/2}}\zeta \mleft(\frac32\mright)\beta \mleft(\frac32\mright) \approx 1.85624. \qedhere
    \end{equation*}
\end{remark}

The improved upper bound for the expected value of the inradius of the planar Brownian convex hull from Theorem~\ref{thm:ub_exp_inr} is not easily sharpened within the same framework because the argument already uses essentially all information available from the two coordinate ranges. In particular, it improves on the cruder area bound \(\sfr\le \sqrt{\sfA/\pi}\) by exploiting two independent one-dimensional constraints simultaneously, rather than a single scalar observable. To do substantially better, one would need geometric information beyond the side lengths of the enclosing rectangle, for example finer information about the shape of the convex hull or about correlations between different extremal directions. Such information is much harder to encode analytically, whereas the present method succeeds precisely because the coordinate ranges are explicit, independent, and admit Feller's series expansion.

Our next set of bounds deal with the inverse inradius process defined in \eqref{eq:gen_inv_proc}. This process was first studied in \cite{CPS}, where the authors showed that
\begin{equation*}
	2 \le \E[\Theta^{\sfr}] \le 83.4.
\end{equation*}
We improve on both of these bounds in the following result.
\begin{theorem}\label{thm:up_low_exp_inv_inr}
    It holds that
    \begin{equation*}
        2.61378 \le \E[\Theta^{\sfr}] \le 18.7460.
    \end{equation*}
\end{theorem}
While the lower bound follows relatively easily, again using the bounding rectangle from Lemma~\ref{lem:r_bound_range}, the proof of the upper bound uses a refinement of the two-stage construction introduced in~\cite{CPS} and is much more involved. As already mentioned, this construction relies on exact formulas for the moments of the exit time of planar Brownian motion from the interior of a carefully calibrated hyperbola, i.e.\ the hourglass domain. In principle, these exit time moments, as functions of the starting position, can be calculated using the generalized Dynkin's formula. This amounts to recursively solving Poisson’s equation on the interior of the domain with Dirichlet boundary conditions; see~\cite[\S 2]{Dynkin_lemma}. However, the unbounded nature of the hourglass domain complicates matters, so we instead opt for a first-principles approach that uses It\^{o}'s lemma.

Next we give upper and lower bounds on the variances of the inradius and its inverse process. We stress here that the lower bounds we obtained should be treated more as explicit rigorous positive bounds, and not so much as sharp estimates. The precise result for the variance of the inradius is as follows.
\begin{theorem}\label{thm:up_low_var_inr}
    It holds that
    \begin{equation*}
        2.7 \cdot 10^{-6} \le \Var(\sfr) \le 0.315520.
    \end{equation*}
\end{theorem}
The analogous result for the variance of the inverse inradius process is formulated in the following theorem.
\begin{theorem}\label{thm:up_low_var_inv_inr}
    It holds that
    \begin{equation*}
        5.2 \cdot 10^{-11} \le \Var(\Theta^{\sfr}) \le 709.599.
    \end{equation*}
\end{theorem}
Both proofs of the lower bounds rely on some basic ideas like Chebyshev's inequality, and bounding rectangles. However, for the upper bound in Theorem~\ref{thm:up_low_var_inr}, we build upon the already nontrivial approach used for the upper bound of $\E[\sfr]$ in Theorem~\ref{thm:ub_exp_inr}. Moreover, for the upper bound in Theorem~\ref{thm:up_low_var_inv_inr} we need to further develop the aforementioned two-stage construction used in the proof of the upper bound for $\E[\Theta^{\sfr}]$ in Theorem~\ref{thm:up_low_exp_inv_inr}.

When it comes to the already known results about the variance of the other geometric functionals, there are again some explicit results, and some bounds. An exact expression for $\Var(\sfP)$ can be found in~\cite{Wade_Xu-SPA}, and an exact expression for the variance of the perimeter of the convex hull of a standard planar Brownian bridge was shown in~\cite{goldman}. The authors in~\cite{Wade_Xu-SPA} also found bounds for $\Var(\sfA)$, and for the variance of the area of the convex hull spanned by the time-space trajectory of a standard one-dimensional Brownian motion. As the authors state in~\cite{Wade_Xu-SPA}, the main interest of the lower bounds they found was that they are positive, which established the non-degeneracy of the corresponding random variables. They write that those bounds are certainly not sharp and can surely be improved, although they have been unable to improve any of them sufficiently to warrant reporting the details. They give an idea for the improvement of the lower bound for the variance of the area of the convex hull spanned by the time-space trajectory of a standard one-dimensional Brownian motion, but they don't go into details. As in our paper here, their bounds are very small (on the order of magnitude $10^{-7}$, and $10^{-6}$), but it seems that it is not easy to get better bounds.

It is worth mentioning here that there are some other interesting results in the literature dealing with convex hulls of time-space trajectories of random processes. In the already mentioned paper~\cite{Wade_Xu-SPA}, the authors provide an explicit expression for the expected value of the area of the convex hull spanned by the time-space trajectory of a standard one-dimensional Brownian motion. This result was generalized to the expected volume, in an arbitrary dimension $d$, of the convex hull of the time-space trajectory of a standard $(d-1)$-dimensional Brownian motion; see~\cite{Cygan-Sandric-Sebek}. Many other questions were studied for convex hulls of time-space trajectories of random processes; see the papers~\cite{MR4828847,MR4718444,MR4684056,AHL_2022__5__779_0,MR2898714,Molchanov-Wespi,MR2978134}.

It is important to mention the study of convex hulls of random walks, which is also a very active field of current research. Many different questions are addressed. Most related to our work are the papers dealing with computations of the expected value of geometric functionals of the convex hull. The first paper in this direction was~\cite{Spitzer-Widom}, where the authors found a formula for the expected perimeter; see also~\cite{Baxter}. Later, a number of works concentrated on several closely related questions; see e.g.~\cite{Snyder-Steele, Vysotsky-Zapor, Cygan-Sandric-Sebek, CSSW, Akopyan_Vysotsky, KVZ, LH_W, McR_W, Wade_Xu-PAMS}.

In higher dimensions, the typical geometric functionals studied are the intrinsic volumes; see~\cite{Kabluchko-Zapor-TAMS, Eldan} for the case of the multidimensional Brownian convex hull, and~\cite{Molchanov-Wespi, Molchanov} for the case of convex hulls spanned by L\'{e}vy processes. Many of the results about particular geometric functionals and their inverse processes from~\cite{CPS} have also been generalized to higher dimensions in~\cite{high_dim_hulls}.

We accompany each of our main results with the corresponding simulation results. More precisely, each quantity for which we show rigorous bounds, we also estimate via the Monte Carlo method. Since we are dealing with inradius, the trickiest part of the simulations was to find the center of the largest circle inscribed inside the convex hull, the so-called Chebyshev's center. All the details about finding the Chebyshev's center, and performing simulations in general can be found in~\cite[Section 5]{CPS}, where authors carried out a simulation study and provided Monte Carlo estimates for all the quantities studied in that paper. The summary of the simulation results, as well as the summary of all the main theorems proved in this paper, can be found in the following Table~\ref{Table-simulations}.
\begin{table}[!h]
\begin{center}
\renewcommand{\arraystretch}{1.3}
	\begin{tabular}{ |c|c|c|c|c| } 
		\hline
		Quantity & Lower bound & MC estimate & Upper bound & Corresponding result \\ 
		\hline \hline
		$\E[\sfr]$& $0.393$\tablefootnote{This is the only bound that is not a result from this paper, but from~\cite{CPS}.} & $0.513$ & $0.667652$ & Theorem~\ref{thm:ub_exp_inr} \\
		\hline
		$\E[\Theta^{\sfr}]$& $2.61378$ & $4.176$ & $18.7460$ & Theorem~\ref{thm:up_low_exp_inv_inr} \\
		\hline
		$\Var[\sfr]$ & $2.7 \cdot 10^{-6}$ & $8.4 \cdot 10^{-3}$ & $0.315520$ & Theorem~\ref{thm:up_low_var_inr} \\
		\hline
		$\Var(\Theta^{\sfr})$ & $5.2\cdot 10^{-11}$ & $2.234$ & $709.599$ & Theorem~\ref{thm:up_low_var_inv_inr}\\
		\hline
	\end{tabular}
	\renewcommand{\arraystretch}{1}
	\vspace{0.3cm}
	\caption{Our bounds on the mean and the variance of the inradius and the inverse inradius process, compared with estimated values coming from the Monte Carlo simulations.}
	\label{Table-simulations}
\end{center}
\end{table}

\begin{remark}
All numerical constants appearing in the bounds below are reported to six decimal places. This convention is not intrinsic, i.e. the individual components of the estimates are known, or can be computed, to arbitrary precision, and increasing the working precision would lead to marginally sharper numerical bounds than the ones reported in Table~\ref{Table-simulations}. Since these improvements are insignificant for the conclusions and would reduce readability, we keep a uniform six-decimal precision throughout.
\end{remark}

The rest of the article is organized as follows. In Section~\ref{sec:hour_glass} we study the exit time of standard planar Brownian motion from the hourglass domain. In Section~\ref{sec:ub_exp_inr} we provide the upper bound on $\E[\sfr]$. Section~\ref{sec:bb_exp_inv_inr} deals with bounds on $\E[\Theta^{\sfr}]$. In Section~\ref{sec:var_bounds} we provide bounds on the variances related to both the inradius and inverse inradius process.

\section{Exit time from the hourglass domain}\label{sec:hour_glass}

Fix the parameters $m>0$ and $b>0$, and consider the hyperbola
\begin{equation}\label{eq:hyperbola}
x^2-\frac{y^2}{m^2}=b^2,
\end{equation}
with asymptotes $y=\pm m x$ and horizontal axis intercepts of $\pm b$. The \emph{hourglass domain} $\mathbb{H}_{m,b}$ is the region strictly between the two hyperbola branches, that is,
\begin{equation}\label{eq:hourglass}
\mathbb{H}_{m,b}:=\mleft\{(x,y)\in\mathbb{R}^2: x^2-\frac{y^2}{m^2}<b^2\mright\}.
\end{equation}
See Figure~\ref{fig:hyperbola} for a depiction of $\mathbb{H}_{m,b}$ with $m=1$ and $b=1$.

\begin{figure}[h!]
\centering
\includegraphics[scale=1]{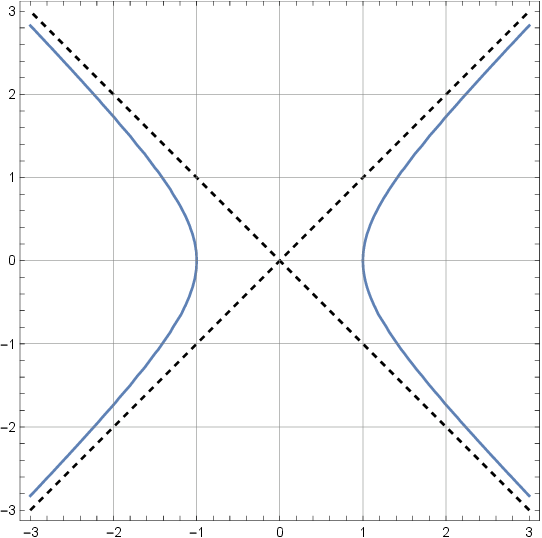}
\caption{The hyperbola from~\eqref{eq:hyperbola} with $m=1$ and $b=1$, and dashed lines for asymptotes. The hourglass domain $\mathbb{H}_{1,1}$ is the region between the two hyperbola branches.}
\label{fig:hyperbola}
\end{figure}

Define the \emph{first exit time} of $\bm{W}$ from $\mathbb{H}_{m,b}$ by
\begin{equation}\label{eq:exit_time}
\tau_{m,b}:=\inf\big\{t\ge  0: \bm{W}(t)\notin \mathbb{H}_{m,b}\big\}.
\end{equation}
Note that $\tau_{m,b}$ is a stopping time with respect to the natural filtration of $\bm{W}$. We need formulas for the first two moments of $\tau_{m,b}$ as functions of the starting point of $\bm{W}$. We begin with a lemma that gives a precise condition on $m>0$ that will ensure these moments are finite. Our proof relies on the fact that the domain $\mathbb{H}_{m,b}$ is \emph{star-like with center} $\boldsymbol{0}$. This means that for every $\boldsymbol{z}\in \mathbb{H}_{m,b}$, the line segment connecting $\boldsymbol{0}$ to $\boldsymbol{z}$ is contained in $\mathbb{H}_{m,b}$. Of particular importance is the \emph{opening angle} $\alpha_m$ of $\mathbb{H}_{m,b}$, which is defined as the angle between the upper asymptote rays. A straightforward calculation gives
\begin{equation}\label{eq:opening_angle}
\alpha_m=2\arctan (1/m).
\end{equation}

\begin{lemma}\label{lem:integrability}
Let the exit time $\tau_{m,b}$ and opening angle $\alpha_m$ be defined as in~\eqref{eq:exit_time} and~\eqref{eq:opening_angle}, respectively, with parameters $m>0$ and $b>0$. Then for all powers $p\in\R$ and all starting points $(x,y)\in \mathbb{H}_{m,b}$, we have 
\[
\E _{(x,y)}[\tau_{m,b}^p]<\infty~\text{ if and only if }~p<\frac{\pi}{2\alpha_m}.
\]
In particular, $\E _{(x,y)}[\tau_{m,b}^p]<\infty$ for all $m>0$ when $p\leq\frac{1}{2}$, while for $p>\frac{1}{2}$, we have 
\[
\E _{(x,y)}[\tau_{m,b}^p]<\infty~\text{ if and only if }~m>\cot\left(\frac{\pi}{4p}\right).
\]
\end{lemma}

\begin{proof}
First note that $\E _{(x,y)}[\tau_{m,b}^p]$ is finite for some $(x,y)\in \mathbb{H}_{m,b}$ if and only if $\E _{(x,y)}[\tau_{m,b}^p]$ is finite for all $(x,y)\in \mathbb{H}_{m,b}$; see~\cite[Eq.~(3.13)]{Burkholder}. Hence, the particular starting point of $\bm{W}$ is irrelevant as long as it starts from the interior of $\mathbb{H}_{m,b}$. Next, we apply \cite[Thm.~2]{Markowsky}. This theorem establishes an equivalent condition for the integrability of the $p$th power of the exit time from a \emph{spiral-like} domain of order $\sigma\geq 0$ with center $\boldsymbol{0}$, in terms of $\sigma$ and a quantity analogous to the maximum opening angle of the domain. As noted in \cite[\S3]{Markowsky}, a star-like domain with center $\boldsymbol{0}$ is also a spiral-like domain of order $0$ with center $\boldsymbol{0}$. It is easy to see that $\mathbb{H}_{m,b}$ is a star-like domain with center $\boldsymbol{0}$, hence the theorem applies in this case with $\sigma=0$ and opening angle $\alpha_m$. This establishes the first if and only if statement, from which the rest of the lemma can be deduced using the formula for $\alpha_m$.
\end{proof}

\begin{remark}
The integrability condition from Lemma~\ref{lem:integrability} is closely related to Spitzer's integrability condition for the exit time of wedges; see~\cite[Thm.~2]{Spitzer}.
\end{remark}


\subsection{First moment of the exit time}

\begin{lemma}\label{lem:torsion}
Consider the exit time $\tau_{m,b}$ of the hourglass domain $\mathbb{H}_{m,b}$, as defined by~\eqref{eq:exit_time} and~\eqref{eq:hourglass}. Then for any $m>1$ and $b>0$, we have
\begin{equation}\label{eq:defn_u(x,y)_expectation}
\E _{(x,y)}[\tau_{m,b}]=\frac{m^2(b^2-x^2)+y^2}{m^2-1}, \quad \text{for all }(x,y)\in\mathbb{H}_{m,b}.
\end{equation}
\end{lemma}

\begin{proof}
We start by defining
\begin{equation}\label{eq:u_formula}
    u(x,y) = \frac{m^2(b^2-x^2)+y^2}{m^2-1},
\end{equation}
for notational convenience. Since $u$ is twice continuously differentiable, we can use It\^{o}'s lemma to write the almost sure equality
\begin{equation}\label{eq:Ito}
u\big(\bm{W}(t)\big)=u\big(\bm{W}(0)\big)+\int_0^t \nabla u\big(\bm{W}(s)\big)\cdot\D{\bm{W}(s)}+\frac{1}{2}\int_0^t\Delta u\big(\bm{W}(s)\big)\D{s}, \quad \text{for all }t\ge  0.
\end{equation}

The stochastic integral term appearing in~\eqref{eq:Ito} can be expanded as 
\begin{equation}\label{eq:martingales}
\int_0^t \nabla u\big(\bm{W}(s)\big)\cdot\D{\bm{W}(s)}=\frac{-2m^2}{m^2-1}\int_0^t W_1(s)\D{W_1(s)}+\frac{2}{m^2-1}\int_0^t W_2(s)\D{W_2(s)},\quad \text{for all }t \ge 0.
\end{equation}
Since both integrands are square-integrable, each term on the right-hand side of~\eqref{eq:martingales} is a genuine martingale with respect to the natural filtration of $\bm{W}$. The Riemann integral term in~\eqref{eq:Ito} can also be simplified by writing
\begin{equation}\label{eq:Laplacian}
\frac{1}{2}\int_0^t\Delta u\big(\bm{W}(s)\big)\D{s}=\frac{1}{2}\int_0^t\frac{-2m^2+2}{m^2-1}\D{s}=-t,\quad \text{for all }t \ge 0.
\end{equation}
To apply the optional stopping theorem to~\eqref{eq:martingales}, we need to work with the sequence of bounded stopping times $T_n:=\tau_{m,b}\wedge n$, with $n\in\mathbb{N}$. Now, taking expectations of both sides of~\eqref{eq:Ito} at time $T_n$ and using~\eqref{eq:Laplacian} results in 
\begin{equation}\label{eq:prelimit}
\E _{(x,y)}\mleft[u\big(\bm{W}(T_n)\big)\mright]=u(x,y)-\E _{(x,y)}[T_n],\quad \text{for all }n \in\mathbb{N}.
\end{equation}

At this juncture, we'd like to show that the left-hand side of~\eqref{eq:prelimit} converges to $0$ as $n\to\infty$. Towards this end, note from~\eqref{eq:hyperbola} that $(x,y)\in\partial\mathbb{H}_{m,b}$ implies $u(x,y)=0$, that is, $u$ vanishes on the boundary of the hourglass. Moreover, the assumption $m>1$ together with Lemma~\ref{lem:integrability} implies that $\tau_{m,b}<\infty$ almost surely. Now it follows from the path continuity of $\bm{W}$ that 
\begin{equation}\label{eq:almost_sure}
\lim_{n\to\infty}u\big(\bm{W}(T_n)\big)=0~~\text{almost surely.}
\end{equation}

We need to combine this almost sure convergence with uniform integrability to get the desired $0$ limit on the left-hand side of~\eqref{eq:prelimit}. We do this by establishing the existence of $p>1$ satisfying
\begin{equation}\label{eq:sup_condition}
\sup_{n\in\mathbb{N}}\E _{(x,y)}\Big[\mleft|u\big(\bm{W}(T_n)\big)\mright|^p\Big]<\infty.
\end{equation}

Since $u(\bm{W}(T_n))$ is a polynomial of degree $2$ in the components of $\bm{W}(T_n)$, there exists a constant $c_{m,b}>0$, depending only on the parameters, such that 
\begin{equation}\label{eq:polynomial_growth}
\mleft|u\big(\bm{W}(T_n)\big)\mright|\le  c_{m,b}\mleft(1+W_1(T_n)^2+W_2(T_n)^2\mright).
\end{equation}
For any $p>1$, we can use~\eqref{eq:polynomial_growth} with Jensen's inequality to write
\[
\mleft|u\big(\bm{W}(T_n)\big)\mright|^p \le  3^{p-1}c_{m,b}^p\mleft(1+|W_1(T_n)|^{2p}+|W_2(T_n)|^{2p}\mright),\quad \text{for all }n \in\mathbb{N}.
\]
Hence, the condition~\eqref{eq:sup_condition} and the desired uniform integrability will follow if we can find $p>1$ such that 
\begin{equation}\label{eq:UI_bound}
\sup_{n\in\mathbb{N}}\E _{(x,y)}\big[|W_1(T_n)|^{2p}\big]<\infty\quad \text{and}\quad \sup_{n\in\mathbb{N}}\E _{(x,y)}\big[|W_2(T_n)|^{2p}\big]<\infty.
\end{equation}

Since the $W_1$ and $W_2$ components of $\bm{W}$ are independent, they are both one-dimensional Brownian motions in the natural filtration of $\bm{W}$. In particular, the components are continuous martingales with respect to that filtration. This allows us to use the Burkholder--Davis--Gundy inequality~\cite[Thm.~3.3.28]{Karatzas_Shreve} at the stopping times $\{T_n\}_{n\ge  1}$. More specifically, for any $p>0$, there exists a positive constant $c_p$, depending only on $p$, such that for all $n\in\mathbb{N}$ we have
\begin{equation}\label{eq:BDG}
\E _{(x,y)}\mleft[\mleft(\sup_{0\le  t\le  T_n}|W_1(t)-x|\mright)^{2p}\mright]\le  c_p\,\E _{(x,y)}\mleft[T_n^p\mright].
\end{equation}
Of course,~\eqref{eq:BDG} also holds with $|W_1(t)-x|$ replaced by $|W_2(t)-y|$.

Assuming that $2p>1$, we can use the triangle and Jensen's inequality to write
\begin{equation}
|W_1(T_n)|^{2p}\le \big(|W_1(T_n)-x|+|x|\big)^{2p}\le  2^{2p-1}\mleft(|W_1(T_n)-x|^{2p}+|x|^{2p}\mright).\label{eq:pre_BDG}
\end{equation}
After taking expectations and then supremums over $n\in\mathbb{N}$ on both sides of~\eqref{eq:pre_BDG}, we can use~\eqref{eq:BDG} along with monotone convergence to write
\begin{equation}\label{eq:BDG_bound}
\sup_{n\in\mathbb{N}}\E _{(x,y)}\big[|W_1(T_n)|^{2p}\big]\le  2^{2p-1}\mleft(c_p\,\E _{(x,y)}\mleft[\tau_{m,b}^p\mright]+|x|^{2p}\mright).
\end{equation}

Next, we appeal to Lemma~\ref{lem:integrability} to show that the right-hand side of~\eqref{eq:BDG_bound} is finite. To see this, note that $m>1$ implies that $\alpha_m<\pi/2$, where $\alpha_m$ is the opening angle defined in~\eqref{eq:opening_angle}. This in turn implies that $\pi/(2\alpha_m)>1$. Hence, $(1,\pi/(2\alpha_m))$ is a nonempty interval, and any $p$ therein will satisfy $\E_{(x,y)}[\tau_{m,b}^p]<\infty$. Since the same argument clearly works for the $Y$ component, we have shown that~\eqref{eq:UI_bound}, and therefore~\eqref{eq:sup_condition}, hold for some $p>1$. Now letting $n\to\infty$ on both sides of~\eqref{eq:prelimit} while using uniform integrability and~\eqref{eq:almost_sure} on the left-hand side and monotone convergence on the right-hand side results in $\E _{(x,y)}[\tau_{m,b}]=u(x,y)$ as desired.
\end{proof}


\subsection{Second moment of the exit time}

\begin{lemma}\label{lem:second_moment}
Consider the exit time $\tau_{m,b}$ of the hourglass domain $\mathbb{H}_{m,b}$, as defined by~\eqref{eq:exit_time} and~\eqref{eq:hourglass}. Then for any $m>1+\sqrt{2}$ and $b>0$, we have
\begin{equation*}
\E _{(x,y)}\mleft[\tau_{m,b}^2\mright] =\frac{\E _{(x,y)}[\tau_{m,b}]}{3(m^4-6m^2+1)}\mleft(\frac{m^2(m^2-5)}{m^2-1}\Big((5m^2-1)b^2-(m^2-1)x^2\Big)+(5m^2-1)y^2\mright)
\end{equation*}
for all $(x,y)\in\mathbb{H}_{m,b}$, where $\E_{(x,y)}[\tau_{m,b}]$ is explicitly calculated in~\eqref{eq:defn_u(x,y)_expectation}.
\end{lemma}

\begin{proof}
The proof of Lemma~\ref{lem:second_moment} closely resembles that of Lemma~\ref{lem:torsion}, so we omit some of the more repetitive details. We start by defining
\begin{equation*}
    v(x,y) = \frac{u(x,y)}{3(m^4-6m^2+1)}\mleft(\frac{m^2(m^2-5)}{m^2-1}\Big((5m^2-1)b^2-(m^2-1)x^2\Big)+(5m^2-1)y^2\mright),
\end{equation*}
where $u(x,y)$ is the formula for the first moment of $\tau_{m,b}$ that was established in Lemma~\ref{lem:torsion}. Notice that neither of the denominators in the expression for $v(x,y)$ can be zero when $m>1+\sqrt{2}$. This is easy to verify by writing 
\begin{equation*}
m^4-6m^2+1=\mleft(m^2-3+2\sqrt{2}\mright)\mleft(m^2-3-2\sqrt{2}\mright)=\mleft(m^2-\big(1-\sqrt{2}\big)^2\mright)\mleft(m^2-\big(1+\sqrt{2}\big)^2\mright).
\end{equation*}

Similarly to the proof of Lemma~\ref{lem:torsion}, we can use It\^{o}'s lemma to write the almost sure equality
\begin{equation}\label{eq:Ito_v}
v\big(\bm{W}(t)\big)=v\big(\bm{W}(0)\big)+\int_0^t \nabla v\big(\bm{W}(s)\big)\cdot\D{\bm{W}(s)}+\frac{1}{2}\int_0^t\Delta v\big(\bm{W}(s)\big)\D{s},\quad \text{for all }t \ge 0.
\end{equation}
Moreover, since the partial derivatives $v_x(W_1(s),W_2(s))$ and $v_y(W_1(s),W_2(s))$ are polynomials in the components of $\bm{W}(s)$, both of which are one-dimensional Brownian motions adapted to the natural filtration of $\bm{W}$, it follows that the stochastic integral term appearing on the right-hand side of~\eqref{eq:Ito_v} is a genuine martingale with respect to this filtration.

For the Riemann integral term, a direct calculation shows that 
\begin{equation}\label{eq:v_Laplacian}
\frac{1}{2}\Delta v(x,y)=-2u(x,y),\quad \text{for all }(x,y)\in\mathbb{R}^2,
\end{equation}
where $u(x,y)$ is given in~\eqref{eq:u_formula}. Hence, taking the expected value of both sides of~\eqref{eq:Ito_v} evaluated at the bounded stopping times $T_n:=\tau_{m,b}\wedge n$ and using~\eqref{eq:v_Laplacian} results in
\begin{equation}\label{eq:v_prelimit}
\E _{(x,y)}\mleft[v\big(\bm{W}(T_n)\big)\mright]=v(x,y)-2\E _{(x,y)}\mleft[\int_0^{T_n}u\big(\bm{W}(s)\big)\D{s}\mright],\quad \text{for all }n \in\mathbb{N}.
\end{equation}

At this point we would like to take the limit as $n\to\infty$ of both sides of~\eqref{eq:v_prelimit}. We start with the term on the left-hand side. We know from Lemma~\ref{lem:torsion} that the function $u$ vanishes on the boundary of $\mathbb{H}_{m,b}$. Since the expression for $v(x,y)$ has a factor of $u(x,y)$, we must have $v$ vanish on the boundary of $\mathbb{H}_{m,b}$ as well. Similarly to the proof of Lemma~\ref{lem:torsion}, we can use this fact along with Lemma~\ref{lem:integrability} and path continuity to argue that $v(\bm{W}(T_n))$ converges to $0$ almost surely as $n\to\infty$. To conclude that the left-hand side of~\eqref{eq:v_prelimit} converges to $0$, we can again follow the proof of Lemma~\ref{lem:torsion}. However, now we are dealing with $v(\bm{W}(T_n))$, which is a polynomial of degree $4$ in the components of $\bm{W}(T_n)$. In this case, Lemma~\ref{lem:integrability} shows that 
\[
m>1+\sqrt{2}=\cot\mleft(\frac{\pi}{4\cdot 2}\mright)
\]
is precisely the condition we need in order to invoke the Burkholder--Davis-Gundy inequality and establish uniform integrability.

It remains to identify the limit of the expected value of the Riemann integral term in~\eqref{eq:v_prelimit}. Since the function $u$ is positive on $\mathbb{H}_{m,b}$ when $m>1$, it follows from monotone convergence that
\begin{equation}\label{eq:u_monotone}
\lim_{n\to\infty}2\E _{(x,y)}\mleft[\int_0^{T_n}u\big(\bm{W}(s)\big)\D{s}\mright]=2\E _{(x,y)}\mleft[\int_0^{\tau_{m,b}}u\big(\bm{W}(s)\big)\D{s}\mright].
\end{equation} 
Moreover, by~\eqref{eq:u_formula} and the strong Markov property, we have that 
\begin{equation}\label{eq:strong_markov}
\E _{(x,y)}\mleft[(\tau_{m,b}-s)\1_{\{s<\tau_{m,b}\}}\middle| \mathcal{F}_s\mright]=u\big(\bm{W}(s)\big)\1_{\{s<\tau_{m,b}\}}~~\text{almost surely.}
\end{equation}
We can use~\eqref{eq:strong_markov} with Tonelli's theorem twice, to write
\begin{align}
2\E _{(x,y)}\mleft[\int_0^{\tau_{m,b}}u\big(\bm{W}(s)\big)\D{s}\mright] &=2\int_0^\infty \E _{(x,y)}\mleft[u\big(\bm{W}(s)\big)\1_{\{s<\tau_{m,b}\}}\mright]\D{s}\nonumber \\
&=2\int_0^\infty\E _{(x,y)}\Bigg[\E _{(x,y)}\mleft[(\tau_{m,b}-s)\1_{\{s<\tau_{m,b}\}}\middle|\mathcal{F}_s\mright]\Bigg]\D{s}\nonumber \\
&=2\E _{(x,y)}\mleft[\int_0^{\tau_{m,b}}(\tau_{m,b}-s)\D{s}\mright]=\E _{(x,y)}\mleft[\tau_{m,b}^2\mright].\label{eq:second_moment}
\end{align}

Now we can combine~\eqref{eq:u_monotone} with~\eqref{eq:second_moment} to conclude that 
\[
\lim_{n\to\infty}2\E _{(x,y)}\mleft[\int_0^{T_n}u\big(\bm{W}(s)\big)\D{s}\mright]=\E _{(x,y)}\mleft[\tau_{m,b}^2\mright].
\]
Since we have established that the left-hand side of~\eqref{eq:v_prelimit} converges to $0$, this shows that $\E _{(x,y)}[\tau_{m,b}^2]=v(x,y)$.
\end{proof}


\section{Upper bound on $\E[\sfr]$}\label{sec:ub_exp_inr}
In this section the main goal is to prove the explicit upper bound from Theorem~\ref{thm:ub_exp_inr}, namely that $\E[\sfr] \le 2(4-\sqrt2)\pi^{-5/2}\zeta\mleft(3/2\mright)\beta\mleft(3/2\mright) \le 0.667652$. The proof is split into some technical propositions. 
\begin{remark}
    Monte Carlo estimation of the value $\E[\sfr]$ is $0.513$.
\end{remark}

The main approach relies on bounding $\sfr$ in terms of the ranges $R_1$ and $R_2$ of the Brownian motions $W_1$ and $W_2$ which are defined by
\begin{equation}\label{eq:def_of_range}
    R_i\coloneqq \sup_{0\le t\le 1} W_i(t)-\inf_{0\le t\le 1} W_i(t), \quad \text{for }\, i=1,2.
\end{equation}
More precisely, we prove the following lemma that is going to be used throughout the paper.

\begin{lemma}\label{lem:r_bound_range}
Let the ranges $R_1$ and $R_2$ be as in \eqref{eq:def_of_range}. It holds that
\begin{equation*}
    \sfr \le \min\{R_1, R_2\}/2, \quad \text{a.s.}
\end{equation*}
\end{lemma}
\begin{proof}
    Denote by $m_i = \inf_{0\le t \le 1}W_i(t)$, and by $M_i = \sup_{0\le t \le 1}W_i(t)$, $i = 1, 2$. Since the Brownian path is contained in the axis-parallel rectangle $[m_1,M_1]\times [m_2,M_2]$, its convex hull $\rmH$ is also contained in that rectangle. Therefore, every closed disk contained in $\rmH$ must also be contained in that rectangle. The largest disk contained in a rectangle of side lengths $R_1$ and $R_2$ has radius $\min\{R_1,R_2\}/2$. Hence, $\sfr\le \min\{R_1,R_2\}/2$ a.s. Notice also that taking expectations yields $\E[\sfr]\le \E[\min\{R_1,R_2\}]/2$.
\end{proof}

Throughout this subsection, we use the notation
\begin{equation}\label{eq:overline_phi_defn}
    \overline\Phi(u)=\int_u^\infty (2\pi)^{-1/2}e^{-v^2/2} \D{v}, \qquad \text{ for all }u>0.
\end{equation}

\begin{proposition}\label{prop:first_bound_red}
Let $\bm W(t)=(W_1(t),W_2(t))$ be a standard planar Brownian motion on the time interval $[0,1]$, and let $R_1$ and $R_2$ be the corresponding ranges. Then, for $\overline{\Phi}$ as in~\eqref{eq:overline_phi_defn}, it follows that
\[
\E[\sfr]\le \frac{1}{2} \E[\min\{R_1,R_2\}]= \frac12 \int_0^\infty \p(R_1\ge x)^2\D x=\frac12 \int_0^\infty
\mleft(
8\sum_{n=1}^\infty (-1)^{n+1} n  \overline\Phi(nx)
\mright)^2 \D x\eqqcolon I.
\]
\end{proposition}

\begin{remark}
\noindent\normalfont{(i)} The gain over the area bound $r\le \sqrt{\sfA/\pi}$ comes from using two independent one-dimensional constraints simultaneously. The convex hull must fit inside the random rectangle determined by the two coordinate ranges. In contrast, the area bound uses only the single scalar quantity $\sfA$.
\medskip 

\noindent\normalfont{(ii)} The argument used to prove Proposition~\ref{prop:first_bound_red} extends directly to Brownian motion in \(\mathbb R^d\), for any \(d\ge2\). Let $\bm W=(W_1,\dots,W_d)$ be a standard \(d\)-dimensional Brownian motion on the time interval \([0,1]\), and let $R_i
$, for $i=1,\dots,d$ be the one-dimensional coordinate ranges. If $\sfr_d$ denotes the inradius of the convex hull of the Brownian path, then $\sfr_d\le \min_{1\le i\le d}R_i/2$ a.s. Indeed, the Brownian path is contained in the axis-parallel box $\prod_{i=1}^d [m_i,M_i]$, for $m_i\coloneqq \inf_{0\le t\le 1}W_i(t)$, and $M_i\coloneqq \sup_{0\le t\le 1}W_i(t)$, and hence so is its convex hull. Any Euclidean ball contained in this box has radius at most half the shortest side length, which is \(\min_{1 \le i \le d} R_i/2\).

Taking expectations yields $\E[\sfr_d]\le \E[\min_{1\le i\le d}R_i]/2$. Since the coordinate processes \(W_1,\dots,W_d\) are independent standard one-dimensional Brownian motions, the random variables \(R_1,\dots,R_d\) are iid. Therefore, for any \(x\ge0\), we have $\p (\min_{1\le i\le d}R_i\ge x)
=
\prod_{i=1}^d \p(R_i\ge x)
=
\p(R\ge x)^d$, where \(R\) denotes the range of a standard one-dimensional Brownian motion on \([0,1]\). As in the proof of $d=2$, it follows that 
\[
\E[\sfr_d]
\le
\frac12\int_0^\infty
\bigg(
8\sum_{n=1}^\infty (-1)^{n+1}n  \overline\Phi(nx)
\bigg)^d \D x.
\]
Thus, the method extends to every dimension \(d\ge2\), reducing the problem to a one-dimensional integral involving the tail distribution of the one-dimensional Brownian range.

What appears to be special to the planar case is the subsequent exact evaluation of this integral in terms of classical special values of \(\zeta\) and \(\beta\). In higher dimensions, the same reduction remains valid and useful for effective upper bounds, but one should not expect an equally simple closed-form evaluation.
\end{remark}

\begin{proposition}\label{prop:square-tail-integral}
For $\overline\Phi(x)
$ as in~\eqref{eq:overline_phi_defn}, define $ S(x)\coloneqq 8\sum_{n=1}^\infty (-1)^{n+1}n  \overline\Phi(nx)$ for all $x>0$, and $I(n,m)\coloneqq \int_0^\infty \overline\Phi(nx)  \overline\Phi(mx)\D x$ for all $m,n \ge 1$. Then, it follows that
\begin{equation}\label{eq:abel-previous}
16\lim_{r\to1^-}\sum_{m,n=1}^\infty (-1)^{m+n}mn  I(n,m)  r^{m^2+n^2}
=
\frac{4-\sqrt2}{\pi^{5/2}}
\zeta \mleft(\frac32\mright)\beta \mleft(\frac32\mright),
\end{equation}
and consequently
\[
I=\frac12\int_0^\infty S(x)^2\D x
=
\frac{2(4-\sqrt2)}{\pi^{5/2}}
\zeta \mleft(\frac32\mright)\beta \mleft(\frac32\mright)
< 0.667652.
\]
\end{proposition}

\begin{proof}[Proof of Proposition~\ref{prop:first_bound_red}]
Since $W_1$ and $W_2$ are independent standard one-dimensional Brownian motions, their respective ranges
$R_1$ and $R_2$ are independent and identically distributed. For any nonnegative random
variable $Y$, it holds that $\E[Y]=\int_0^\infty \p(Y\ge x)\D x$. Hence, for $Y=\min\{R_1,R_2\}$, it holds that
\[
\E[\min\{R_1,R_2\}]
=
\int_0^\infty \p(\min\{R_1,R_2\}\ge x)\D x.
\]
By independence of $R_1$ and $R_2$, we see that
\[
\p(\min\{R_1,R_2\}\ge x)
=
\p(R_1\ge x)\p(R_2\ge x)
=
\p(R_1\ge x)^2, \quad \text{ for all }x\ge 0.
\]
Therefore, together with Lemma~\ref{lem:r_bound_range}, it follows that
\begin{equation}\label{eq:range-square-bound}
\E[\sfr] \le \frac12\int_0^\infty \p(R_1\ge x)^2\D x=\frac12\int_0^\infty \p(R\ge x)^2\D x,
\end{equation} where $R_1$ has the same law as the range $R\coloneqq \sup_{0\le t\le 1} W(t)-\inf_{0\le t\le 1} W(t)$ of a standard one-dimensional Brownian motion $W=\{W(t):t \ge 0\}$. By Feller's expansion~\cite[Eq.~(3.7) \&~(3.8)]{Feller1951}, we have that $\p(R\ge x)
=
8\sum_{n=1}^\infty (-1)^{n+1}n  \overline\Phi(nx)$, for all $x>0$. Hence,
\[
\E[\sfr]
\le
\frac12\int_0^\infty
\mleft(
8\sum_{n=1}^\infty (-1)^{n+1}n  \overline\Phi(nx)
\mright)^2\D x
= I.\qedhere
\]
\end{proof}

\begin{proof}[Proof of Proposition~\ref{prop:square-tail-integral}]
We divide the proof into three steps.\medskip

\noindent\textbf{Step 1:} We start by considering the evaluation of \(I(n,m)\). For \(n,m>0\), we note that
\[
\overline\Phi(nx)=\int_{nx}^\infty \phi(u)\D u,
\quad \text{and}\quad
\overline\Phi(mx)=\int_{mx}^\infty \phi(v)\D v,
\quad \text{ for } \quad 
\phi(t)\coloneqq \frac{1}{\sqrt{2\pi}}e^{-t^2/2}.
\]
By Tonelli's theorem, it follows that
\begin{equation}\label{eq:I(n,m)_polar}
    I(n,m) =
\int_0^\infty
\int_{nx}^\infty \phi(u)\D u
\int_{mx}^\infty \phi(v)\D v\D x =
\int_0^\infty  \int_0^\infty
\min \mleft(\frac{u}{n},\frac{v}{m}\mright)\phi(u)\phi(v)\D u\D v .
\end{equation}
Changing to polar coordinates \(u=r\cos\theta\), \(v=r\sin\theta\) in the first quadrant, and writing $\theta_0\coloneqq \arctan(m/n)$, we obtain that
\[
\min \mleft(\frac{u}{n},\frac{v}{m}\mright)
=
\begin{dcases}
\dfrac{r\sin\theta}{m}, & \text{for }0\le \theta\le \theta_0,\\
\dfrac{r\cos\theta}{n}, & \text{for }\theta_0\le \theta\le \dfrac{\pi}{2}.
\end{dcases}
\]
Hence, using the polar coordinates in~\eqref{eq:I(n,m)_polar}, it holds that
\[
I(n,m)=\frac{1}{2\pi}\int_0^\infty r^2e^{-r^2/2}\D r
\mleft(
\frac1m\int_0^{\theta_0}\sin\theta\D \theta
+
\frac1n\int_{\theta_0}^{\pi/2}\cos\theta\D{\theta}
\mright).
\]
Since $\int_0^\infty r^2e^{-r^2/2}\D r=\sqrt{\pi/2}$, $\cos\theta_0=n/\sqrt{m^2+n^2}$, and $\sin\theta_0=m/\sqrt{m^2+n^2}$, it follows that
\begin{equation}\label{eq:I-explicit}
I(n,m)=\frac{1}{2\sqrt{2\pi}}
\mleft(
\frac1n+\frac1m-\frac{\sqrt{m^2+n^2}}{mn}
\mright).
\end{equation}
\medskip

\noindent\textbf{Step 2:} We now evaluate the quadratically regularized sum corresponding to~\eqref{eq:abel-previous}. Define
\[
T(r)\coloneqq \sum_{m,n=1}^\infty (-1)^{m+n}mn  I(n,m)  r^{m^2+n^2},
\quad \text{ for all } 0<r<1.
\]
Using~\eqref{eq:I-explicit}, we obtain $mn  I(n,m)=(2\sqrt{2\pi})^{-1}(m+n-\sqrt{m^2+n^2})$,
and hence
\begin{align}
T(r)&=\frac{1}{2\sqrt{2\pi}}\Bigl(A(r)-B(r)\Bigr), \quad \text{ for all }0<r<1, \text{ where }\label{eq:T-split}\\
A(r)&\coloneqq \sum_{m,n\ge1}(-1)^{m+n}(m+n)r^{m^2+n^2}, \quad \text{and}\quad 
B(r)\coloneqq \sum_{m,n\ge1}(-1)^{m+n}\sqrt{m^2+n^2}  r^{m^2+n^2}. \nonumber
\end{align}

We first evaluate \(A(r)\) as \(r\to1^-\). The series defining \(A(r)\) is absolutely convergent, since
\begin{equation}\label{eq:abs_sum_A}
\sum_{m,n\ge1}(m+n)r^{m^2+n^2}
\le
2\bigg(\sum_{m\ge1}mr^{m^2}\bigg)\bigg(\sum_{n\ge1}r^{n^2}\bigg)<\infty,
\end{equation} and therefore $A(r)
=
2\big(\sum_{m\ge1}(-1)^m m r^{m^2}\big)
\big(\sum_{n\ge1}(-1)^n r^{n^2}\big)$. By writing \(r=e^{-t}\) with \(t>0\), it follows that
\begin{equation}\label{eq:decomp_A_GF}
A(e^{-t})=2F(t)G(t),
\quad \text{for} \quad 
F(t)\coloneqq \sum_{n\ge1}(-1)^n n e^{-tn^2},
\quad \text{and}\quad 
G(t)\coloneqq \sum_{n\ge1}(-1)^n e^{-tn^2}.
\end{equation} Next, we evaluate \(G(t)\) from~\eqref{eq:decomp_A_GF}. Set $\Theta(t)\coloneqq \sum_{n\in\mathbb Z}(-1)^n e^{-tn^2}=1+2G(t)$. By Poisson summation formula~\cite[Ch.~5, Thm~3.1]{MR1970295} applied to \(x\mapsto e^{-tx^2}e^{i\pi x}\), it follows that
\[
\Theta(t)
=
\sqrt{\frac{\pi}{t}}
\sum_{k\in\mathbb Z} e^{-\pi^2(k-1/2)^2/t}.
\]
Since \((k-1/2)^2\ge 1/4+|k|-1\) for all \(k\in\mathbb Z\), we have $\sum_{k\in\mathbb Z} e^{-\pi^2(k-1/2)^2/t}
\le
2e^{-\pi^2/(4t)}\sum_{j=0}^\infty e^{-\pi^2 j/t}$, and therefore \(\Theta(t)\to0\) as \(t\to0^+\). Hence
\begin{equation}\label{eq:G-limit}
\lim_{t\to0^+}G(t)=-1/2.
\end{equation}

Next, we evaluate \(F(t)\) from~\eqref{eq:decomp_A_GF}. Define $h_t(x)\coloneqq xe^{-tx^2}$ for all $x \ge 1$, implying that $-F(t)=\sum_{n=1}^\infty (-1)^{n+1}h_t(n)$. Since $\lim_{n\to\infty}h_t(2n)
=
\lim_{n\to\infty}h_t'(2n)
=0$, and \(h_t'''\in L^1(1,\infty)\),~\cite[Cor.~4.8]{Lampret2022} with \(p=3\) and \(m=1\), yields
\[
-F(t)
=
h_t(1)-\frac12h_t(2)+\frac14h_t'(2)+R_t, \quad \text{where}\quad |R_t|
\leq
\frac18\int_2^\infty |h_t'''(x)|\D x.
\]
Writing \(h(y)=ye^{-y^2}\), we have that $h_t(x)=t^{-1/2}h(\sqrt t x)$, $h_t'''(x)=t h'''(\sqrt t x)$,
and hence
\[
\int_2^\infty |h_t'''(x)|\D x
\leq
\sqrt t\int_0^\infty |h'''(y)|\D y
\to 0, \quad \text{as }t \to 0^+,
\] since $h'''(y)=(-6+24y^2-8y^4)e^{-y^2}\in L^1(0,\infty)$. Thus, \(R_t \to0\) as $t \to 0^+$, and $h_t(1)\to1$, $h_t(2)/2\to1$, and $h_t'(2)\to1$ as $t \to 0^+$. Consequently, $\lim_{t\to0^+}F(t)
=
-1+1-1/4 =
-1/4$. Combining this with~\eqref{eq:decomp_A_GF} and~\eqref{eq:G-limit}, we
conclude that
\begin{equation}\label{eq:A-limit}
\lim_{r\to1^-}A(r)=\frac14.
\end{equation}

Next we analyze \(B(r)\). Consider the Dirichlet series $E(s)$, defined as 
\begin{align*}
E(s)&\coloneqq \sum_{m,n=1}^\infty \frac{(-1)^{m+n}}{(m^2+n^2)^s},
\quad \text{ and } \quad \widetilde E(s)\coloneqq 
\sum_{(m,n)\in\mathbb Z^2\setminus\{(0,0)\}}
\frac{(-1)^{m+n}}{(m^2+n^2)^s},
\quad \text{for } \Re s>1.
\end{align*}
Since \(\Re s>1\), the series defining \(\widetilde E(s)\) is absolutely convergent, so we may regroup terms according to the decomposition
\[
\mathbb Z^2\setminus\{(0,0)\}
=
\{(m,n)\in\mathbb Z^2:m,n\neq0\}
\cup
\{(m,0):m\neq0\}
\cup
\{(0,n):n\neq0\}.
\]
The off-axis part splits into the four open quadrants, and each quadrant contributes $\sum_{m,n\ge1}(-1)^{m+n}(m^2+n^2)^{-s}=E(s)$, since replacing \(m,n\) by \(|m|,|n|\) leaves both \(m^2+n^2\) and \((-1)^{m+n}\) unchanged. Hence, the total off-axis contribution is \(4E(s)\). On the coordinate axes, we have
\begin{align*}
    \sum_{m\neq0}\frac{(-1)^m}{m^{2s}}
&=
2\sum_{k=1}^\infty \frac{(-1)^k}{k^{2s}}
=
-2\eta(2s), \quad \text{for} \quad 
\eta(s)\coloneqq \sum_{k=1}^\infty \frac{(-1)^{k-1}}{k^s}
=(1-2^{1-s})\zeta(s),
\end{align*}
where $\eta$ is the Dirichlet eta function and $\zeta(s)$ is the Riemann zeta function defined in~\eqref{eq_riemann_zeta_function_def}. Therefore, $\widetilde E(s)=4E(s)-4\eta(2s)$. On the other hand, the twisted lattice sum \(\widetilde E(s)\) has the classical factorization
(see, e.g.,~\cite[(1.7)]{Burrows2022} or~\cite{MR407483})
\begin{equation}\label{eq:defn_beta_fct}
    \widetilde E(s)=-4  \eta(s)\beta(s),
\quad\text{for all } \Re s>1,
\end{equation} where $\beta(s)$ is the Dirichlet beta function defined in~\eqref{eq:defn_beta_diri}. Hence,
\begin{equation}\label{eq:E-factor}
E(s)=\eta(2s)-\eta(s)\beta(s),
\quad \text{for all } \Re s>1.
\end{equation}

Since \(\eta\) and \(\beta\) are entire, the function $s\mapsto  \eta(2s)-\eta(s)\beta(s)$ is entire. By~\eqref{eq:E-factor}, this entire function agrees with the Dirichlet series \(E(s)\) on the half-plane \(\Re s>1\), where the latter converges absolutely. Hence, \(\eta(2s)-\eta(s)\beta(s)\) defines the unique analytic continuation of \(E(s)\) from \(\Re s>1\) to all \(s\in\mathbb C\). We continue to denote this entire continuation by \(E(s)\). In particular, \(E(s)\) is holomorphic at \(s=-1/2\). This is also consistent with the general analytic continuation theory for twisted Epstein zeta functions; see~\cite[Ch.~I, \S5, Thm.~3 \& Eq.~(61)]{SiegelANT}.

Now write \(r=e^{-t}\) with \(t>0\). Then, $B(e^{-t})
=
\sum_{m,n=1}^\infty 
(-1)^{m+n}\sqrt{m^2+n^2}\,
e^{-t(m^2+n^2)}$. Let \(c>3/2\), and note by the Mellin inversion~\cite[Thm~2]{FlajoletGourdonDumas}, that $e^{-x}
=
(2\pi i)^{-1}
\int_{c-i\infty}^{c+i\infty}
\Gamma(w)x^{-w}\D w$ for all $x>0$, and therefore
\[
\sqrt{m^2+n^2}\,e^{-t(m^2+n^2)}
=
\frac{1}{2\pi i}
\int_{c-i\infty}^{c+i\infty}
\Gamma(w)t^{-w}
(m^2+n^2)^{-(w-1/2)}
\D w.
\]
Since \(c>3/2\), it follows directly that $\sum_{m,n=1}^\infty 
(m^2+n^2)^{-(c-1/2)}
<\infty$. Moreover, by Stirling's formula on vertical lines~\cite[Eq.~(5.11.9)]{NIST:DLMF}, it follows that $|\Gamma(c+\mathrm{i} y)| \sim \sqrt{2 \pi}|y|^{c-(1 / 2)} e^{-\pi|y| / 2}$ as $|y| \to \infty$ uniformly for bounded real values of $c$, and hence $\int_{-\infty}^{\infty}
|\Gamma(c+iy)|\D y<\infty$. Thus, Fubini's theorem applies, and summing over \(m,n\ge1\), gives
\begin{equation}\label{eq:B-Mellin-representation}
B(e^{-t})
=
\frac{1}{2\pi i}
\int_{c-i\infty}^{c+i\infty}
\Gamma(w)t^{-w}
E\mleft(w-\frac12\mright)
\D w.
\end{equation} We now apply the converse mapping theorem for Mellin transforms
\cite[Thm~4(i)]{FlajoletGourdonDumas}.
To verify its hypotheses, fix any \(\delta\in(0,1)\). Recall that $E(s)=\eta(2s)-\eta(s)\beta(s)$ gives the entire continuation of \(E\). Hence, $M(w)
\coloneqq
\Gamma(w)E(w-1/2)$ is meromorphic in the strip $-\delta<\Re w<c$,  and, since \(E\) is entire, its only possible pole in this strip is at \(w=0\), coming from the
simple pole of \(\Gamma(w)\). It remains to verify the decay condition required in~\cite[Thm~4(i)]{FlajoletGourdonDumas}. First, Stirling's formula uniformly on bounded vertical strips~\cite[Eq.~(5.11.9)]{NIST:DLMF} gives $|\Gamma(\sigma+\mathrm{i} y)| \sim \sqrt{2 \pi}|y|^{\sigma-(1 / 2)} e^{-\pi|y| / 2}$ as $|y| \to \infty$ uniformly for \(\sigma\) in compact intervals. Convexity bounds for the Riemann zeta function~\cite[Ch.~V, \S5.1]{TitchmarshZeta} imply that \(\zeta(s)\) has at most polynomial growth as \(|\Im s|\to\infty\), uniformly when \(\Re s\) ranges over a fixed compact interval. Likewise, the convexity bound for Dirichlet
\(L\)-functions due to~\cite{Davenport1931}, together with the functional equation for \(L(s,\chi_4)\), implies the corresponding polynomial growth for \(\beta(s)=L(s,\chi_4)\). Since
\[
\eta(s)=(1-2^{1-s})\zeta(s),
\]
and the factor \(1-2^{1-s}\) is bounded when \(\Re s\) ranges over a fixed
compact interval, \(\eta(s)\) and \(\beta(s)\) have at most polynomial growth
as \(|\Im s|\to\infty\), uniformly in every fixed vertical strip. It follows from $E(s)=\eta(2s)-\eta(s)\beta(s)$ that, for every fixed strip \(a\le \Re s\le b\), there exist constants
\(C,N>0\) such that $|E(s)|\le C(1+|\Im s|)^N$. Consequently, uniformly for $-\delta\le\Re w\le\eta$, where \(\eta\in(3/2,c)\) is fixed, we have
\[
\left|
\Gamma(w)
E\mleft(w-\frac12\mright)
\right|
=
O\mleft(
(1+|\Im w|)^N e^{-\pi|\Im w|/2}
\mright), \quad \text{ as } |\Im w| \to \infty.
\]
In particular, $\Gamma(w)E (w-1/2)
=
O(|w|^{-q})$ for every \(q>0\) as \(|\Im w|\to\infty\), uniformly in the above strip. Thus, the decay hypothesis of~\cite[Thm~4(i)]{FlajoletGourdonDumas}
is satisfied. The singular expansion of \(M(w)\) in
\(-\delta<\Re w<\eta\) therefore consists only of the pole at \(w=0\). Since $\mathrm{Res}_{w=0}\Gamma(w)=1$, we have $\mathrm{Res}_{w=0}
[
\Gamma(w)E(w-1/2)
]
=
E(-1/2)$. Here, \(\operatorname{Res}_{w=w_0} f(w)\) denotes the residue of a meromorphic
function \(f\) at \(w_0\), i.e., the coefficient of \((w-w_0)^{-1}\) in the
Laurent expansion of \(f\) around \(w_0\).
Applying
\cite[Thm~4(i)]{FlajoletGourdonDumas}
to~\eqref{eq:B-Mellin-representation}, then yields
\[
B(e^{-t})
=
E\mleft(-\frac12\mright)
+
O(t^\delta),
\qquad t\to0^+.
\]
Since \(\delta\in(0,1)\) was arbitrary, in particular
\begin{equation}\label{eq:B-limit}
\lim_{r\to1^-}B(r)
=
E\mleft(-\frac12\mright).
\end{equation}

Applying~\eqref{eq:E-factor} at \(s=-1/2\), we obtain $E (-1/2)=\eta(-1)-\eta (-1/2)\beta (-1/2)$. Since \(\eta(-1)=1/4\), \(\zeta(-1/2)=-(4\pi)^{-1}\zeta(3/2)\), and
\(\beta(-1/2)=\pi^{-1}\beta(3/2)\), it follows that
\begin{align}
    \eta \mleft(-\frac12\mright)
&=
\bigl(1-2^{3/2}\bigr)\zeta \mleft(-\frac12\mright)
=
\frac{2^{3/2}-1}{4\pi}  \zeta \mleft(\frac32\mright), \text{ and hence}\nonumber\\
E \mleft(-\frac12\mright)
&=
\frac14-\frac{2^{3/2}-1}{4\pi^2}  
\zeta \mleft(\frac32\mright)\beta \mleft(\frac32\mright).\label{eq:E-minus-half-evaluated}
\end{align}

Combining~\eqref{eq:T-split},~\eqref{eq:A-limit},~\eqref{eq:B-limit}, and~\eqref{eq:E-minus-half-evaluated}, we obtain~\eqref{eq:abel-previous}:
\begin{align*}
16\lim_{r\to1^-}T(r)
&=
\frac{8}{\sqrt{2\pi}}
\mleft(
\frac14-E \mleft(-\frac12\mright)
\mright)=
\frac{8}{\sqrt{2\pi}}
\cdot
\frac{2^{3/2}-1}{4\pi^2}  
\zeta\mleft(\frac32\mright)\beta \mleft(\frac32\mright)=
\frac{4-\sqrt2}{\pi^{5/2}}
\zeta\mleft(\frac32\mright)\beta \mleft(\frac32\mright).
\end{align*}
\medskip

\noindent\textbf{Step 3:} In the final step, we will evaluate \(\int_0^\infty S(x)^2\D x\). For \(0<r\le1\) and all $x>0$, define
\[
S_r(x)\coloneqq 8\sum_{n=1}^\infty (-1)^{n+1}n  r^{n^2}  \overline\Phi(nx).
\]For each fixed \(x>0\), the series defining \(S_r(x)\) is absolutely
convergent, uniformly in \(0<r\le1\), since
\[
\sum_{n=1}^\infty n r^{n^2}\overline\Phi(nx)
\le
\sum_{n=1}^\infty n\overline\Phi(nx)
<\infty.
\]
Hence, by dominated convergence for series, $S_r(x)\to S(x)$ as $r \to 1^-$ for every \(x>0\). We next establish an integrable bound for \(S_r(x)\), uniformly in \(0<r\le1\).

First, suppose that \(x\ge 1\), and let $a_n^{(r)}(x)\coloneqq n  r^{n^2}  \overline\Phi(nx)$. We claim that \(n\mapsto a_n^{(r)}(x)\) is decreasing. Indeed, if $h(y)\coloneqq y  \overline\Phi(y)$, then $h'(y)=\overline\Phi(y)-y\phi(y)$. By Mills' ratio, it follows that $\overline\Phi(y)<\phi(y)/y$ for all $y \ge 1$, and hence \(h'(y)<0\), so \(h\) is strictly decreasing. Therefore, it holds that $n\mapsto n  \overline\Phi(nx)=x^{-1}h(nx)$ is decreasing, and so is \(n\mapsto a_n^{(r)}(x)\) because \(0<r\le1\). Since
\(a_n^{(r)}(x)\to0\) as \(n\to\infty\), the Leibniz criterion gives
\begin{equation}\label{eq:Sr-large-x-bound}
|S_r(x)|\le 8  a_1^{(r)}(x)\le 8  \overline\Phi(x),
\quad \text{ for all } x\ge 1,\ 0<r\le1.
\end{equation}

It remains to obtain a uniform bound for \(0<x<1\). Write \(r=e^{-t}\) for 
\(t\ge0\), and set $f_{x,t}(u)\coloneqq u e^{-tu^2}\overline\Phi(xu)$. Since
\(f_{x,t}(u),f'_{x,t}(u)\to0\) as \(u\to\infty\) and
\(f''_{x,t}\in L^1([1,\infty))\),~\cite[Cor.~4.6]{Lampret2022} with $m=1$, implies
\[
\sum_{n=1}^\infty(-1)^{n+1}f_{x,t}(n)
=
f_{x,t}(1)-\frac12f_{x,t}(2)
+\frac14f'_{x,t}(2)+R_{x,t}, \quad \text{where}\quad |R_{x,t}|
\le
\frac5{12}\int_2^\infty |f''_{x,t}(u)|\D u.
\] For \(0<x<1\), the quantities \(f_{x,t}(1)\), \(f_{x,t}(2)\), and
\(f'_{x,t}(2)\) are uniformly bounded in \(x\) and \(t\). Moreover, setting $\lambda=t x^{-2}$, and $G_\lambda(y)=y e^{-\lambda y^2}\overline\Phi(y)$, we have $f''_{x,t}(u)=xG_\lambda''(xu)$, and hence $\int_2^\infty |f''_{x,t}(u)|\D u
\le
\int_0^\infty |G_\lambda''(y)|\D y$. A direct differentiation gives
\begin{equation}\label{eq:defn_g''_lamvda}
    G_\lambda''(y)
=
e^{-\lambda y^2}
\left[
(4\lambda^2y^3-6\lambda y)\overline\Phi(y)
+
\bigl((4\lambda+1)y^2-2\bigr)\phi(y)
\right],
\end{equation}
whose \(L^1(0,\infty)\)-norm is bounded uniformly in \(\lambda\ge0\). Indeed, using \(\overline\Phi(y)\le1/2\),
\(\lambda y^2e^{-\lambda y^2}\le e^{-1}\), $\int_0^\infty y^3e^{-\lambda y^2}\D y=(2\lambda^2)^{-1}$, and $\int_0^\infty ye^{-\lambda y^2}\D y=(2\lambda)^{-1}$
for \(\lambda>0\), together with $\int_0^\infty y^2\phi(y)\D y=1/2$, and $\int_0^\infty\phi(y)\D y=1/2$, we obtain $\sup_{\lambda\ge0}
\int_0^\infty |G_\lambda''(y)|\D y
\le
4+2/e$.

Consequently, since $\sup_{\lambda \ge 0}\|G''_\lambda\|_{L^1(0,\infty )}<\infty$, there exists an absolute constant \(C>0\) such that $|S_r(x)|\le C$ for all $0<x<1$ and $0<r\le1$. Combining this with~\eqref{eq:Sr-large-x-bound}, we obtain
\[
|S_r(x)|
\le
C\1_{(0,1)}(x)
+
8\overline\Phi(x)\1_{[1,\infty)}(x), 
\qquad \text{for all } x>0,\, 0<r\le1.
\]
The right-hand side is in $L^2(0,\infty)$, and since \(S_r(x)\to S(x)\) pointwise as \(r\to1^-\), dominated convergence therefore implies that
\begin{equation}\label{eq:dominated}
\int_0^\infty S(x)^2\D x
=
\lim_{r\to1^-}\int_0^\infty S_r(x)^2\D x.
\end{equation}

We now consider the right-hand side of~\eqref{eq:dominated}. For fixed \(0<r<1\), the series $\sum_{m,n\ge1} mn  r^{m^2+n^2}I(n,m)$ is absolutely convergent. Indeed, by~\eqref{eq:I-explicit}, it follows that $0\le I(n,m)\le (2\sqrt{2\pi})^{-1}(n^{-1}+m^{-1})$, implying that
\[
\sum_{m,n\ge1} mn  r^{m^2+n^2}I(n,m)
\le
\frac{1}{2\sqrt{2\pi}}
\sum_{m,n\ge1}(m+n)r^{m^2+n^2}
<\infty.
\]
Therefore, Fubini's theorem yields
\begin{align}
\int_0^\infty S_r(x)^2\D x
&=
64\sum_{m,n\ge1}(-1)^{m+n}mn  r^{m^2+n^2}
\int_0^\infty \overline\Phi(nx)\overline\Phi(mx)\D x \notag\\
&=
64\sum_{m,n\ge1}(-1)^{m+n}mn  I(n,m)  r^{m^2+n^2}. \label{eq:Sr-square}
\end{align}
Letting \(r\to1^-\) and using~\eqref{eq:abel-previous},~\eqref{eq:dominated}, and~\eqref{eq:Sr-square}, we conclude the proof of the proposition:
\begin{equation*}
\int_0^\infty S(x)^2\D x
=
64\lim_{r\to1^-}\sum_{m,n\ge1}(-1)^{m+n}mn  I(n,m)  r^{m^2+n^2} =
\frac{4(4-\sqrt2)}{\pi^{5/2}}
\zeta \mleft(\frac32\mright)\beta \mleft(\frac32\mright). \qedhere
\end{equation*}
\end{proof}

\begin{proof}[Proof of Theorem~\ref{thm:ub_exp_inr}]
    The proof follows directly from Proposition~\ref{prop:first_bound_red} and Proposition~\ref{prop:square-tail-integral}.
\end{proof}

\section{Bounds on $\E[\Theta^{\sfr}]$}\label{sec:bb_exp_inv_inr}
In this section we prove Theorem~\ref{thm:up_low_exp_inv_inr} which claims that $2.61378 \le \E[\Theta^{\sfr}] \le 18.7460$.

\begin{remark}
    Monte Carlo estimation of the value $\E[\Theta^{\sfr}]$ is $4.176$.
\end{remark}

Let $(R_1(t) : t \ge 0)$ and $(R_2(t) : t \ge 0)$ denote, as in \eqref{eq:def_of_range}, the coordinate-wise range processes, and let $(\Theta_1(y) : y \ge 0)$ and $(\Theta_2(y) : y \ge 0)$ denote the corresponding inverse range processes, i.e.
\begin{equation}\label{eq:Theta_1,2_BM}
    \Theta_i(y) = \inf\{t \ge 0 : R_i(t) > y\}, \quad y \ge 0,\ i = 1,2.
\end{equation}
Note that, for simplicity, we write $\Theta_i$ instead of $\Theta^{R_i}$ (which was the general notation introduced in~\eqref{eq:gen_inv_proc}), and we again write $R_i = R_i(1)$, and $\Theta_i = \Theta_i(1)$, $i = 1, 2$. Using this notation we first prove the following lemma.
\begin{lemma}\label{lem:lb_exp_inv_inr}
    It holds that
    \begin{equation*}
        \E[\Theta^\sfr] \ge 4 \E[\max\{\Theta_1, \Theta_2\}] =2 + \frac{8}{\pi} \int_0^{\infty} \mleft( 1 - \frac{4}{(\cos t + \cosh t)^2} \mright) \frac{1}{t^3} \D t \ge 2.61378.
    \end{equation*}
\end{lemma}
\begin{proof}
    The convex hull of the path up to time $t$ is contained in the axis--aligned rectangle with side lengths $R_1(t)$ and $R_2(t)$, and therefore, by Lemma~\ref{lem:r_bound_range}, its inradius $\sfr(t)$ satisfies $\sfr(t) \le \min\{R_1(t),R_2(t)\}/2$ a.s. Hence, we need to have $\min\{R_1(t), R_2(t)\} > 2$ to be able to have $\sfr(t) > 1$. Therefore
    \begin{equation}\label{eq:Thetar_lb_step1}
        \Theta^{\sfr} \ge \inf\mleft\{t \ge 0 : \min\{R_1(t), R_2(t)\} > 2\mright\} = \max\{\Theta_1(2), \Theta_2(2)\} \stackrel{d}{=} 4 \max\{\Theta_1, \Theta_2\},
    \end{equation}
    where we used the well-known scaling of the inverse range process, which states that $\Theta_1(\lambda y) \stackrel{d}{=} \lambda^2 \Theta_1(y)$ for every $\lambda > 0$ and $y \ge 0$ (see e.g.~\cite{Vallois}). Furthermore, we know that $\E[\Theta_1] = \E[\Theta_2] = 1/2$, which is easily obtained using the Laplace transform of the inverse range process from~\cite{Imhof}. Also, from~\cite[Lemma 1.7]{CPS} we know that
    \begin{equation}\label{eq:inv_range_first_mom}
        \E[\min\{\Theta_1, \Theta_2\}] = \frac{1}{2} - \frac{2}{\pi} \int_0^{\infty} \mleft( 1 - \frac{4}{(\cos t + \cosh t)^2} \mright) \frac{1}{t^3} \D t \approx 0.346554.
    \end{equation}
    Combining this with the trivial relation $\min\{\Theta_1, \Theta_2\} + \max\{\Theta_1, \Theta_2\} = \Theta_1 + \Theta_2$, we get $\E[\max\{\Theta_1, \Theta_2\}] = 1 - \E[\min\{\Theta_1, \Theta_2\}] \approx 0.653446$. Plugging this into~\eqref{eq:Thetar_lb_step1}, we get $\E[\Theta^{\sfr}] \ge 2.61378$.
    \end{proof}

\subsection{Upper bound on expectation of inverse inradius}\label{sec:up_bound_inrad_constr}
We bound the inverse inradius time from above by using a two-stage construction that embeds a triangle with a prescribed minimum inradius into the convex hull of the path of the Brownian motion; see Figure~\ref{fig:construction}. At the instant this embedding is achieved, the inradius of the convex hull of the Brownian path at least the prescribed minimum value, hence the inverse inradius time has already passed. Our construction is a refinement of the method used to prove~\cite[Proposition 1.11]{CPS}, and it significantly improves upon that upper bound. A similar multi-stage construction was used to bound other inverse processes in two and higher dimensions; see~\cite{CPS, high_dim_hulls}.

\begin{proposition}\label{prop:two_stage_inradius_upper_bound}
Fix $a>0$, $b>0$, and $m>1$. Let
\begin{equation}\label{eq:defn_Q_m,b,a}
    Q_{m,b,a}(x)=(x-1)^2\mleft(a^2(x-2)-x\mright)-m^2(x-2)(x^2-b^2),
\end{equation}
and assume that $ Q_{m,b,a}(x)\ge 0$ for all $x \ge b$. Let $(\Theta_1(y) : y\ge 0)$ and $(\Theta_2(y) : y\ge 0)$ be the inverse range processes of the coordinates of $\bm W$ defined in~\eqref{eq:Theta_1,2_BM}, and let $T_a \coloneqq \min\{\Theta_1(2a),\Theta_2(2a)\}$. At time $T_a$, recenter and rotate the path so that $\bm W (T_a)=A=(0,a)$ and the line segment of length $2a$ already contained in the convex hull is sent to the segment with endpoints $A=(0,a)$ and $B=(0,-a)$. Let $\tau_{m,b,a}$ be the first exit time from $\mathbb H_{m,b}$ of this transformed Brownian motion after $T_a$. Then,
\[
    \Theta^{\mathsf r}
    \le
    \min\{\Theta_1(2a),\Theta_2(2a)\}
    +
    \tau_{m,b,a} \quad \text{a.s.}
\]
\end{proposition}

\begin{proof}
At time $T_a$, at least one of the coordinates has range $2a$. Hence, the convex hull of the Brownian path up to time $T_a$ contains a line segment of length at least $2a$, and $\bm{W}(T_a)$ is one endpoint of such a segment. Replacing this segment by a subsegment if necessary, we obtain a segment of length exactly $2a$ contained in $\mathcal{H}(T_a)$ with $\bm{W}(T_a)$ as one endpoint. Let $\mathfrak R$ be an orthogonal transformation, measurable with respect to $\mathcal F_{T_a}$, chosen so that the affine map $z \mapsto \mathfrak R\bigl(z-\bm W(T_a)\bigr)+A$ sends the selected segment of length $2a$, with endpoint $\bm W(T_a)$, onto the vertical segment with endpoints $A=(0,a)$ and $B=(0,-a)$, sending $\bm W(T_a)$ to $A$. Define $\widetilde{\bm W}(t)
\coloneqq
\mathfrak R\bigl(\bm W(T_a+t)-\bm W(T_a)\bigr)+A$, for all $t \ge 0$. Then $\widetilde{\bm W}(0)
=
\mathfrak R(0)+A
=
A$, since $\mathfrak R$ is linear. By the strong Markov property together with rotational invariance, $\widetilde{\bm W}$ is a planar Brownian motion started from $A=(0,a)$, independent of $\mathcal F_{T_a}$. Define the stopping time
\[
    \tau_{m,b,a}
    =
    \inf\{t\ge 0:\widetilde{\bm{W}}(t)\notin \mathbb H_{m,b}\},
\]
and write $ C=\widetilde{\bm{W}}(\tau_{m,b,a})$ for the corresponding exit point. Since rigid motions preserve convex hulls and inradii, the convex hull of the original Brownian path up to time
$T_a+\tau_{m,b,a}$ contains, after the above recentering and rotation, the triangle with vertices $A$, $B$, and $C$, see Figure~\ref{fig:construction}.

It remains to show that every possible such triangle has inradius at least $1$. By symmetry of the hourglass domain with respect to both coordinate axes, it is enough to consider exit points on the upper half of the right branch of $\partial \mathbb H_{m,b}$. Such a point has the form $C=(x,y(x))$ where $y(x)=m\sqrt{x^2-b^2}$ for all $x \ge b$. The area of the triangle $ABC$ is $ax$, and its semiperimeter is
\[
    \frac12\left(
        2a+\sqrt{x^2+(y(x)-a)^2}
        +\sqrt{x^2+(y(x)+a)^2}
    \right).
\]
Therefore, the inradius of the triangle $ABC$ is given by
\begin{equation}\label{eq:inradius}
    r(x)
    =
    \frac{2ax}
    {2a+\sqrt{x^2+(y(x)-a)^2}+\sqrt{x^2+(y(x)+a)^2}} .
\end{equation}
We claim that the $Q$ condition $Q_{m,b,a}(x)\ge 0$ for all $x \ge b$ implies $ r(x)\ge 1$ for all $x \ge b$. We divide the proof of this implication into four steps.\medskip

\noindent\textbf{Step 1:} Here we show that the $Q$ condition forces $b>2$. At $x=b$, it follows from the condition that
\[
Q_{m,b,a}(b)=(b-1)^2(a^2(b-2)-b)\geq 0.
\]
With the possible exception of $b=1$, this immediately rules out $0<b<2$. However, we cannot have $b=1$ or $b=2$ either, since in both of these cases setting $x=2\geq b$ results in $Q_{m,b,a}(2)=-2$, which violates the condition. Therefore, the $Q$ condition implies $b>2$.\medskip

\noindent\textbf{Step 2:} Here we establish an inequality that compares the $y$-coordinate of the upper half of the right hyperbola branch to an intermediate quantity $H(x)$, defined by
\[
H(x)=(x-1)^2\mleft(a^2-\frac{x}{x-2}\mright),
\]
that will facilitate a comparison to the inradius $r(x)$ in a later step. Towards this end, we start with the $Q$ condition 
\[
(x-1)^2\mleft(a^2(x-2)-x\mright)-m^2(x-2)(x^2-b^2)\geq 0,
\]
and then use the hyperbola equation to substitute $y(x)^2$ for the factor of $m^2(x^2-b^2)$ appearing in the second term on the left-hand side. This results in 
\[
(x-1)^2\mleft(a^2(x-2)-x\mright)-(x-2)y(x)^2\geq 0.
\]
From \textbf{Step 1} we know that $x\geq b>2$, hence we can divide both sides by $x-2$ to arrive at
\[
y(x)^2\leq (x-1)^2\mleft(a^2-\frac{x}{x-2}\mright).
\]
Therefore, the $Q$ condition implies $y(x)\leq \sqrt{H(x)}$ for all $x\geq b$.\medskip

\noindent\textbf{Step 3:} Here we convert the bound from \textbf{Step 2} into an upper bound for the sum of the lengths of the sides $AC$ and $BC$ in the triangle $ABC$ depicted in Figure~\ref{fig:construction}. For fixed $x>2$, let $D(v)$ denote the sum of the distances from the points $A$ and $B$ to the point $(x,v)$, namely, 
\[
D(v)=\sqrt{x^2+(v-a)^2}+\sqrt{x^2+(v+a)^2}.
\]
To see that $D(v)$ is a nondecreasing function of $v\geq 0$, it suffices to consider the sign of the derivative
\[
D'(v)=\frac{v-a}{\sqrt{x^2+(v-a)^2}}+\frac{v+a}{\sqrt{x^2+(v+a)^2}}.
\]
It is clear that $D'(v)\geq 0$ when $v\geq a$. When $0\leq v<a$, we can get the same conclusion by rewriting $D'(v)$ as
\[
D'(v)=-\frac{a-v}{\sqrt{x^2+(a-v)^2}}+\frac{a+v}{\sqrt{x^2+(a+v)^2}},
\]
and then applying the increasing function $w\mapsto w/\sqrt{x^2+w^2}$ to both sides of the inequality $a-v\leq a+v$ to obtain
\[
\frac{a-v}{\sqrt{x^2+(a-v)^2}}\leq \frac{a+v}{\sqrt{x^2+(a+v)^2}}.
\]

Since the function $D$ is nondecreasing on the interval $[0,\infty)$, we can apply it to both sides of the inequality from \textbf{Step 2} to yield $0\leq D(y(x))\leq D(\sqrt{H(x)})$. Moreover, the equality $D(\sqrt{H(x)})=2a(x-1)$ for $x>2$ can be verified with a straightforward but tedious calculation that starts by squaring both sides. Therefore, the $Q$ condition implies $D\big(y(x)\big)\leq 2a(x-1)$ for all $x\geq b$.\medskip

\noindent\textbf{Step 4:} Here we finally arrive at the desired inradius inequality. Recalling the expression for $r(x)$ from~\eqref{eq:inradius}, we can transform the inequality 
\[
\sqrt{x^2+\big(y(x)-a\big)^2}+\sqrt{x^2+\big(y(x)+a\big)^2}\leq 2a(x-1)
\]
from \textbf{Step 3} into $r(x)\geq 1$. Therefore, the $Q$ condition implies  $r(x)\geq 1$ for all $x\geq b$.

Moreover, the inequality $r(x)\ge 1$ is strict for the random exit point almost surely. Indeed, equality can occur only when \(Q_{m,b,a}(x)=0\), and since \(Q_{m,b,a}\) is a nonzero polynomial, this corresponds to only finitely many points on \(\partial\mathbb H_{m,b}\), each of which is hit by planar Brownian motion with probability zero.

Hence, at time \(T_a+\tau_{m,b,a}\), the convex hull of the Brownian path contains a triangle whose inradius is strictly larger than \(1\) almost surely. Since the inradius of a convex set is monotone under inclusion, the convex hull itself has inradius strictly larger than \(1\) at that time almost surely. Consequently,
\[
    \Theta^{\mathsf r}
    \le
    T_a+\tau_{m,b,a}
    =
    \min\{\Theta_1(2a),\Theta_2(2a)\}
    +
    \tau_{m,b,a}
    \qquad \text{a.s.} \qedhere
\]
\end{proof}

Define now the first-moment objective function
\begin{equation}\label{eq:defn_object_funct_mean_inradi}
    \Psi_1(m,b,a) = 4M_\Theta(1) a^2+\frac{m^2 b^2+a^2}{m^2-1},
\end{equation}
over the feasible region $\mathfrak{F}(1)$, where $M_\Theta(1)\coloneqq \E[\min\{\Theta_1(1),\Theta_2(1)\}]\le 0.346555
$ (see~\eqref{eq:inv_range_first_mom} or~\cite[Lemma 1.7]{CPS}) and 
\begin{equation}\label{eq:feasible_region_F(r)}
    \mathfrak{F}(p)\coloneqq \{(m,b,a) \in \R^3: m>\cot(\pi/(4p)), \, b>0,\, a >0, \, Q_{m,b,a}(x) \ge 0 \text{ for all } x \ge b \}, \quad \text{ for }p >1/2.
\end{equation}

\begin{figure}[h!]
\centering
\begin{tikzpicture}[scale=1.05]
  \def\a{2.6}
  \def\m{2.0}
  \def\b{1.8}
  \def\Ymax{5.0} 
  \draw[->,thick] (-0.2,0) -- (9.0,0) node[below right] {$x$};
  \draw[->,thick] (0,-5.2) -- (0,5.2) node[left] {$y$};
  \pgfmathsetmacro{\xcut}{\Ymax/\m}
  \draw[dashed,gray,thick] (-0.8,{-\m*0.8}) -- (\xcut,\Ymax) node[above] {$y=mx$};
  \draw[dashed,gray,thick] (-0.8,{\m*0.8}) -- (\xcut,{-\Ymax}) node[below] {$y=-mx$};
  \draw[ultra thick,blue,domain=-5:5,samples=600,variable=\yy]
    plot({sqrt(\a*\a + (\yy*\yy)/(\m*\m))},{\yy});
  \draw[ultra thick,blue,domain=-5:5,samples=600,variable=\yy]
    plot({-sqrt(\a*\a + (\yy*\yy)/(\m*\m))},{\yy});
  \filldraw[red] (0,\b) circle (1.8pt) node[above left, yshift=-5pt] {$A=(0,a)$};
  \filldraw[red] (0,-\b) circle (1.8pt) node[below left, yshift=5pt] {$B=(0,-a)$};
  \draw[red,very thick] (0,-\b) -- (0,\b);
  \def\yC{2.6}
  \pgfmathsetmacro{\xC}{sqrt(\a*\a + (\yC*\yC)/(\m*\m))}
  \filldraw[black] (\xC,\yC) circle (1.8pt) node[right] {$C$};
  \draw[very thick] (0,\b) -- (\xC,\yC) -- (0,-\b) -- cycle;
  \filldraw[black] (\a,0) circle (1.5pt) node[below right] {$(b,0)$};
  \foreach \Y in {-4,-2,0,2,4} \draw (-0.1,\Y) -- (0.1,\Y);
  \foreach \X in {0,2,4,6,8} \draw (\X,-0.1) -- (\X,0.1);
\end{tikzpicture}
\caption{The hourglass domain $\mathbb{H}_{m,b}$ with the line segment $AB$ and the triangle $ABC$ for a generic exit point $C$ on the right branch.}
\label{fig:construction}
\end{figure}
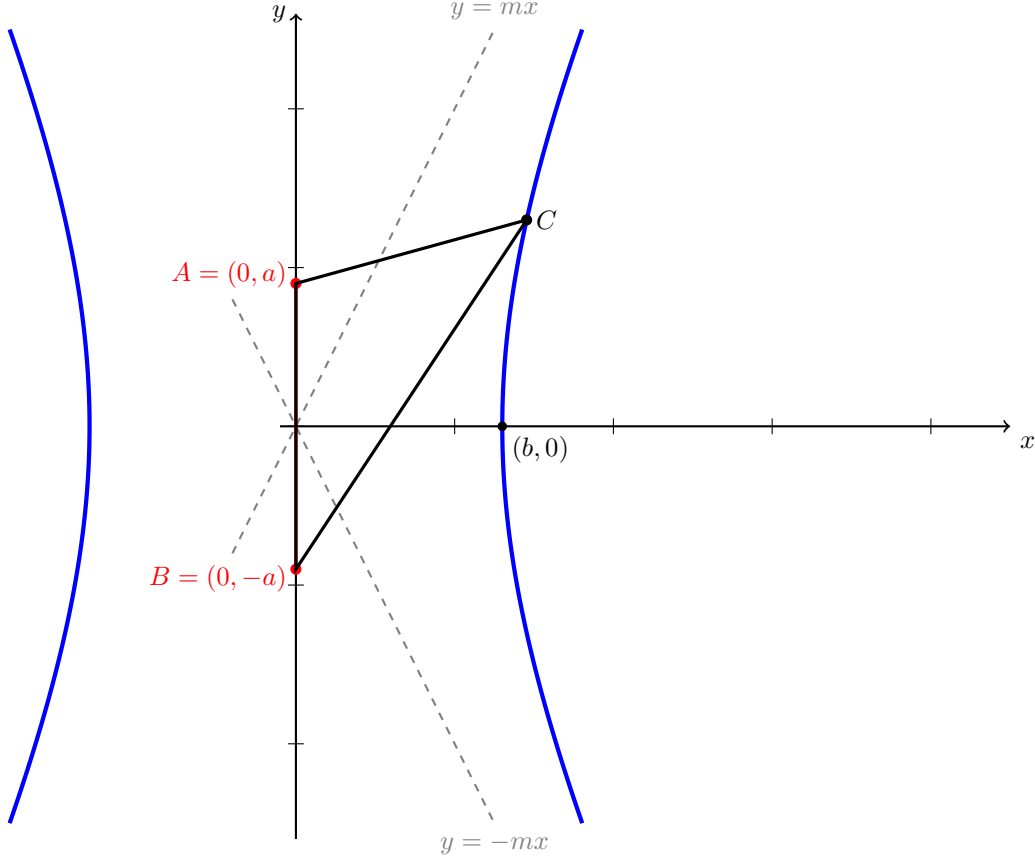

\begin{lemma}\label{lem:bound_first_moment_Inverse_range}
For $\Psi_1$ as in~\eqref{eq:defn_object_funct_mean_inradi} and $\mathfrak{F}(1)$ as in~\eqref{eq:feasible_region_F(r)}, we have
\begin{equation*}
    \E[\Theta^\sfr] \le \inf_{(m,b,a)\in \mathfrak{F}(1)} \Psi_1(m,b,a)\le 18.7460.
\end{equation*} 
\end{lemma}

\begin{proof}
We note by Proposition~\ref{prop:two_stage_inradius_upper_bound} and Lemma~\ref{lem:torsion}, that
$\E[\Theta^\sfr]\le 4a^2\E[\min\{\Theta_1,\Theta_2\}]+\E_{(0,a)}[\tau_{m,b}]=\Psi_1(m,b,a)$ for all $(m,b,a) \in \mathfrak{F}(1)$. Hence, it follows directly, that $\E[\Theta^\sfr] \le \inf_{(m,b,a) \in \mathfrak{F}(1)} \Psi_1(m,b,a)$. Taking $m = 1.908007$, $a = 2.484217$ and $b = 2.386748$ gives the desired bound. It is straightforward to check that this point is indeed in the feasibility region $\mathfrak{F}(1)$, by checking the value of $Q_{m, b, a}(x)$ at $x = b$, the two critical points in $(b, \infty)$, and as $x \to \infty$.
\end{proof}

\begin{proof}[Proof of Theorem~\ref{thm:up_low_exp_inv_inr}]
The proof of the lower bound follows directly from Lemma~\ref{lem:lb_exp_inv_inr}, and the proof of the upper bound follows directly from Lemma~\ref{lem:bound_first_moment_Inverse_range}.
\end{proof}

\section{Variance bounds for inradius and related inverse process}\label{sec:var_bounds}

In this section we develop bounds for the variance of $\sfr$ and $\Theta^{\sfr}$. As we already stressed, lower bounds should be treated more as explicit rigorous positive bounds, not as sharp estimates. The main idea for both lower bounds is to use Chebyshev's inequality. Recall that for a random variable $X$, and for any $\varepsilon \ge 0$ we have
\begin{equation*}
    \Var(X) = \E[(X - \E X)^2] \ge \varepsilon^2 \p(|X - \E X| \ge \varepsilon).
\end{equation*}
If $\E X \ge 0$, taking $\varepsilon = \alpha \E X$, for $\alpha > 0$, we obtain
\begin{equation}\label{eq:cheb_with_alpha}
    \Var(X) \ge \alpha^2 (\E X)^2 \mleft( \p(X \le (1 - \alpha)\E X) + \p(X \ge (1 + \alpha)\E X) \mright),
\end{equation}
and our lower bounds use whichever of the latter two probabilities is most convenient. For the upper bound on $\Var(\sfr)$, we further develop the approach used in the proof of Theorem~\ref{thm:ub_exp_inr}, and for the upper bound on $\Var(\Theta^{\sfr})$, we use the two-stage construction from Lemma~\ref{lem:bound_first_moment_Inverse_range}.

\subsection{Two--sided bounds on $\Var(\sfr)$}
In this subsection we prove Theorem~\ref{thm:up_low_var_inr} which claims that $2.7 \cdot 10^{-6} \le \Var(\sfr) \le 0.315520$.
\begin{remark}
    Monte Carlo estimation of the value $\Var(\sfr)$ is $8.4 \cdot 10^{-3}$.
\end{remark}

Denote by $F_R(x)$ the cumulative distribution function of the one-dimensional Brownian range. The explicit form of $F_R$ is known from~\cite{Jovalekic}, namely
\begin{equation}\label{eq:cdf-range}
    F_R(x) = 8\sum_{m=1}^\infty \mleft(\frac{1}{x^2}+\frac{1}{(2m-1)^2\pi^2}\mright) \exp  \mleft(-\frac{(2m-1)^2\pi^2}{2x^2}\mright), \quad \text{for all } x > 0.
\end{equation}

We first prove the following lemma.
\begin{lemma}\label{lem:lb_var_inr}
    For every $\alpha \in (0, 1)$ it holds that
    \begin{equation*}
        \Var(\sfr) \ge \alpha^2 \big(\E[\sfr]\big)^2 \mleft[ 1 - \big( 1 - F_R(2(1 - \alpha)\E[\sfr]) \big)^2 \mright].
    \end{equation*}
    Furthermore, taking only the first term of the series for $F_R(x)$, setting $\alpha = 0.1$, and using the lower bound $\E[\sfr]\ge 0.393$ from~\cite{CPS}, we obtain
    \begin{equation*}
        \Var(\sfr) \ge 2.7 \cdot 10^{-6}.
    \end{equation*}
\end{lemma}
\begin{proof}
    From~\eqref{eq:cheb_with_alpha} we clearly have
    \begin{equation}\label{eq:var_r-lb-step1}
        \Var(\sfr) \ge \alpha^2 (\E[\sfr])^2 \p(\sfr \le (1 - \alpha) \E[\sfr]).
    \end{equation}
    Recall that $R_1$ and $R_2$ are the coordinate-wise ranges, and, by Lemma~\ref{lem:r_bound_range}, that the inradius satisfies $\sfr \le \min\{R_1,R_2\}/2$ a.s. Consequently, for any $x>0$,
    \begin{equation*}
        \p(\sfr \le x) \ge \p(\min\{R_1, R_2\} \le 2x) = 1 - \p(R_1 > 2x, R_2 > 2x).
    \end{equation*}
    Since $R_1$ and $R_2$ are independent and identically distributed,
    \begin{equation}\label{eq:var_r-lb-step2}
        \p(\sfr \le x) \ge 1 - \mleft( 1 - F_R(2x) \mright)^2,
    \end{equation}
    where $F_R$ denotes the cumulative distribution function from~\eqref{eq:cdf-range}. Setting $x=(1-\alpha)\E[\sfr]$ in~\eqref{eq:var_r-lb-step2}, and plugging this into~\eqref{eq:var_r-lb-step1} yields
    \begin{equation*}
        \Var(\sfr) \ge \alpha^2 \big(\E[\sfr]\big)^2 \mleft[ 1 - \big( 1 - F_R(2(1 - \alpha)\E[\sfr]) \big)^2 \mright].
    \end{equation*}
    Using the lower bound $\E[\sfr]\ge 0.393$ from~\cite{CPS}, we obtain
    \begin{equation}\label{eq:var_r-lb-final}
        \Var(\sfr) \ge \alpha^2 (0.393)^2 \mleft[ 1 - \mleft( 1 - F_R(2(1 - \alpha)0.393) \mright)^2 \mright].
    \end{equation}
Notice that all terms of the series in~\eqref{eq:cdf-range} are positive. Truncating the series after finitely many terms therefore yields a lower bound on $F_R$, which is sufficient for~\eqref{eq:var_r-lb-final}, since the right--hand side is monotone increasing in $F_R$. Optimizing the right-hand side of~\eqref{eq:var_r-lb-final} over $\alpha\in(0,1)$ and retaining only the first term of the series for $F_R$ (the second term already being smaller than $10^{-30}$ at the optimizer $\alpha \approx 0.1$) gives the explicit bound
\begin{equation}\label{eq:var_r-lb}
    \Var(\sfr) \ge 0.1^2 \cdot (0.393)^2 \mleft[ 1 - \mleft(  1 - 8 \mleft( \frac{1}{(2\cdot 0.9 \cdot 0.393)^2} + \frac{1}{\pi^2} \mright) \exp{\mleft( -\frac{\pi^2}{2(2\cdot 0.9 \cdot 0.393)^2}\mright)}\mright)^2\mright] \ge 2.7 \cdot 10^{-6}. 
\end{equation}
\end{proof}

\begin{proposition}\label{prop:min-range-square}
For $m \ge 1$, let  
\begin{equation}
    A_m \coloneqq 
\frac{e^{-(2m-1)\pi}}{(1+e^{-(2m-1)\pi})^2}
+\frac{e^{-(2m-1)\pi}}{\pi(2m-1)(1+e^{-(2m-1)\pi})}
+\frac{\log(1+e^{-(2m-1)\pi})}{\pi^2(2m-1)^2}.
\end{equation} Then, it follows that
\begin{equation}\label{eq:bound_r_squared}
    \E[\sfr^2] \le \frac{1}{4}\E  \mleft[\min\{R_1,R_2\}^2\mright]
=
\frac{1}{2}\int_0^\infty x  \p(R_1\ge x)^2 \D x= 2\log 2
-16\sum_{m=1}^\infty A_m.
\end{equation} Moreover, since $A_m >0$ for all $m \ge 1$, it holds that 
\begin{equation}\label{eq:approx_min_suqared}
    \E  \mleft[\min\{R_1, R_2\}^2\mright]
\;\le\;
8\log 2-64\sum_{m=1}^5 A_m < 1.87988,
\end{equation} yielding that 
\begin{equation}\label{eq:var_bound}
    \Var(\sfr) < 0.315520.
\end{equation}
\end{proposition}

\begin{proof}
Similarly to the proof of Proposition~\ref{prop:first_bound_red}, we know by Lemma~\ref{lem:r_bound_range} and a change of variables, that
\begin{equation}\label{eq:equal_first_proof}
    \E[\sfr^2] \le \frac{1}{4}\E  \mleft[\min\{R_1, R_2\}^2\mright]
=
\frac{1}{4}\int_0^\infty \p(\min\{R_1, R_2\}^2\ge u)\D u = \frac{1}{4}\int_0^\infty 2x \p(R \ge x)^2 \D x,
\end{equation}
which proves the first identity in~\eqref{eq:bound_r_squared}. We now evaluate explicitly the quantity in~\eqref{eq:equal_first_proof}. Let $F_R$ and $f_R$ denote the cumulative distribution
function and probability density function of $R$. Recall that the explicit expression for $F_R$ is already given in~\eqref{eq:cdf-range}. The following classical formula holds for $x>0$ by~\cite[Eq.~(3.6)]{Feller1951}:
\begin{equation}\label{eq:density-range}
    f_R(x) = \sqrt{\frac{8}{\pi}} \sum_{n=1}^\infty (-1)^{n-1}n^2 e^{-n^2x^2/2}.
\end{equation}

Since the variable $\min\{R_1, R_2\}$ is nonnegative, it follows that $\E[\min\{R_1, R_2\}^2]=\int_0^\infty 2x  \p(\min\{R_1, R_2\}\ge x)\D x$. Because $R_1,R_2$ are independent and identically distributed,
\[
\p(\min\{R_1, R_2\}\ge x)=\p(R_1\ge x,R_2\ge x)=\p(R\ge x)^2=(1-F_R(x))^2.
\]
Hence, $\E[\min\{R_1, R_2\}^2]
=
\int_0^\infty 2x  (1-F_R(x))^2\D x$. On the other hand, the distribution function of $\min\{R_1, R_2\}$ is $F_{\min\{R_1, R_2\}}(x)=1-(1-F_R(x))^2$, so, since $R$ has density $f_R$, the random variable $\min\{R_1, R_2\}$ has density
\[
f_{\min\{R_1, R_2\}}(x)=F_{\min\{R_1, R_2\}}'(x)=2f_R(x)(1-F_R(x)).
\]
Therefore, 
\[
\E[\min\{R_1, R_2\}^2]
=
\int_0^\infty x^2 f_{\min\{R_1, R_2\}}(x)\D x
=
2\int_0^\infty x^2 f_R(x)(1-F_R(x))\D x,
\] and hence, since $\E[R^2]=4 \log 2$ by~\cite[Eq.~(1.4)]{Feller1951}, it follows that
\[
\E[\min\{R_1, R_2\}^2]
=
2\E[R^2]
-
2\int_0^\infty x^2 f_R(x)F_R(x)\D x=8 \log(2)
-
2\int_0^\infty x^2 f_R(x)F_R(x)\D x.
\]

Next, write $q_m \coloneqq 2m-1$ and $\beta_m \coloneqq q_m^2 \pi^2 /2$. Substituting~\eqref{eq:density-range} and~\eqref{eq:cdf-range} into $I \coloneqq \int_0^\infty x^2 f_R(x)F_R(x)\D x$, and using Fubini's theorem, we obtain
\begin{equation}\label{eq:I_reduction_double_sum}
I
=
8\sqrt{\frac{8}{\pi}}
\sum_{m,n \ge 1} (-1)^{n-1}n^2
\int_0^\infty
\mleft(1+\frac{x^2}{q_m^2\pi^2}\mright)
\exp  \mleft(-\frac{n^2x^2}{2}-\frac{q_m^2\pi^2}{2x^2}\mright)\D x.
\end{equation}
We now use the ensuing standard identities, valid for $a,b>0$, which follow by~\cite[Eq.~3.471(9) \&~8.469(3)]{MR2360010}, and by differentiating the first identity below w.r.t. $a$: 
\begin{equation}\label{eq:bessel2}
\int_0^\infty e^{-ax^2-b/x^2}\D x
=
\frac{\sqrt{\pi}}{2\sqrt a}  e^{-2\sqrt{ab}}, \quad \text{and}\quad \int_0^\infty x^2 e^{-ax^2-b/x^2}\D x
=
\frac{\sqrt{\pi}}{4a^{3/2}}(1+2\sqrt{ab})e^{-2\sqrt{ab}}.
\end{equation}Applying~\eqref{eq:bessel2} with $a= n^2/2$, and $b=q_m^2\pi^2/2$, we have $2\sqrt{ab}=nq_m\pi$, and hence
\begin{equation}\label{eq:clas_equli_normal}
\begin{aligned}
\int_0^\infty e^{-n^2x^2/2-q_m^2\pi^2/(2x^2)}\D x
&=
\frac{\sqrt{2\pi}}{2n}e^{-nq_m\pi}, \quad \text{ and }
\\
\int_0^\infty x^2 e^{-n^2x^2/2-q_m^2\pi^2/(2x^2)}\D x
&=
\frac{\sqrt{2\pi}}{2n^3}(1+nq_m\pi)e^{-nq_m\pi}.
\end{aligned}
\end{equation}
Substituting~\eqref{eq:clas_equli_normal} into~\eqref{eq:I_reduction_double_sum} yields
\[
I
=
16\sum_{m=1}^\infty\sum_{n=1}^\infty (-1)^{n-1}
e^{-nq_m\pi}
\mleft(
n+\frac{1}{\pi q_m}+\frac{1}{\pi^2 q_m^2 n}
\mright).
\]

Let now $t_m\coloneqq e^{-q_m\pi}=e^{-(2m-1)\pi}$. Since $q_m=2m-1\ge 1$, we have $0<t_m=e^{-(2m-1)\pi}<1$. Hence, all power series below converge absolutely. First, by the geometric series formula, we have that $\sum_{n=0}^\infty (-t_m)^n=(1+t_m)^{-1}$. Subtracting the $n=0$ term gives
\begin{equation}\label{eq:geom_sum_1}
    \sum_{n=1}^\infty (-1)^{n-1} t_m^n
=
-\sum_{n=1}^\infty (-t_m)^n
=
1-\frac{1}{1+t_m}
=
\frac{t_m}{1+t_m}.
\end{equation} Next, since the power series $\sum_{n=0}^\infty (-t)^n=1/(1+t)$ has convergence radius $1$, it may be differentiated termwise for $t \in (-1,1)$. Thus, for $|t|<1$, $\sum_{n=1}^\infty n(-1)^n t^{n-1}
=
-(1+t)^{-2}$. Evaluating at $t=t_m$ and multiplying by $-t_m$, we obtain
\begin{equation}\label{eq:geom_sum_2}
    \sum_{n=1}^\infty (-1)^{n-1}n t_m^n
=
\frac{t_m}{(1+t_m)^2}.
\end{equation} Finally, again because the geometric series has radius of convergence $1$, it may be integrated termwise on compact subintervals of $(-1,1)$. So, since $\sum_{n=0}^\infty (-u)^n=(1+u)^{-1}$ for all $|u|<1$, we may integrate from $0$ to $t_m$, to get $\sum_{n=0}^\infty (-1)^n\int_0^{t_m}u^n\D u
=
\int_0^{t_m}(1+u)^{-1} \D u$. That is, $\sum_{n=0}^\infty (-1)^nt_m^{n+1}(n+1)^{-1}
=
\log(1+t_m)$.
Reindexing with $k=n+1$ yields
\begin{equation}\label{eq:geom_sum_3}
    \sum_{k=1}^\infty (-1)^{k-1}\frac{t_m^k}{k}
=
\log(1+t_m).
\end{equation}

Altogether, by~\eqref{eq:geom_sum_1},~\eqref{eq:geom_sum_2} and~\eqref{eq:geom_sum_3}, it holds that
\begin{align*}
\sum_{n=1}^\infty (-1)^{n-1}n t_m^n
&=
\frac{t_m}{(1+t_m)^2}, \quad \sum_{n=1}^\infty (-1)^{n-1} t_m^n =
\frac{t_m}{1+t_m}, \quad \text{and}\quad \sum_{n=1}^\infty (-1)^{n-1}\frac{t_m^n}{n}=
\log(1+t_m),
\end{align*}
concluding that
\[
I
=
16\sum_{m=1}^\infty
\mleft[
\frac{t_m}{(1+t_m)^2}
+\frac{t_m}{\pi q_m(1+t_m)}
+\frac{\log(1+t_m)}{\pi^2 q_m^2}
\mright].
\]
Since $\E[\min\{R_1, R_2\}^2]
=
2\E[R^2]-2I$, and $\E[R^2]=4\log 2$, we conclude~\eqref{eq:bound_r_squared}, i.e., that
\begin{align*}
    \E[\min\{R_1, R_2\}^2]
&=
8\log 2
-64\sum_{m=1}^\infty A_m, \quad \text{ where}\\
A_m &\coloneqq 
\frac{e^{-(2m-1)\pi}}{(1+e^{-(2m-1)\pi})^2}
+\frac{e^{-(2m-1)\pi}}{\pi(2m-1)(1+e^{-(2m-1)\pi})}
+\frac{\log(1+e^{-(2m-1)\pi})}{\pi^2(2m-1)^2}.
\end{align*}

Since each $A_m>0$, taking any finite truncation of $I$ yields an upper bound, and hence~\eqref{eq:approx_min_suqared} follows:
\[
\E[\min\{R_1, R_2\}^2]
\le
8\log 2-64\sum_{m=1}^5 A_m=1.8798749620\ldots 
<1.8798749621.
\]

Finally, we will now conclude the variance bound~\eqref{eq:var_bound}. Indeed, using lower bound for $\E[\sfr]$ from~\cite{CPS} we get $\Var(\sfr) = \E[\sfr^2] - \E[\sfr]^2 \le \E[\sfr^2] - (0.393)^2$. Hence, together with~\eqref{eq:approx_min_suqared}, it follows that 
\begin{equation*}
     \Var(\sfr) <  \frac{1.8798749621}{4} - (0.393)^2 = 0.315519740525.\qedhere
\end{equation*}
\end{proof}

\begin{proof}[Proof of Theorem~\ref{thm:up_low_var_inr}]
The proof of the lower bound follows directly from Lemma~\ref{lem:lb_var_inr}, and the proof of the upper bound follows directly from Proposition~\ref{prop:min-range-square}. 
\end{proof}

\subsection{Two--sided bounds on $\Var(\Theta^{\sfr})$}
In this subsection we prove Theorem~\ref{thm:up_low_var_inv_inr} which claims $5.2 \cdot 10^{-11} \le \Var(\Theta^{\sfr}) \le 709.599$.

\begin{remark}
    Monte Carlo estimation of the value $\Var(\Theta^{\sfr})$ is $2.234$.
\end{remark}

\begin{lemma}\label{lem:lb_var_inv_inr}
    For every $\alpha \in (0, 1)$ it holds that
    \begin{equation*}
        \Var(\Theta^\sfr) \ge \alpha^2 \big( \E[\Theta^\sfr] \big)^2 \mleft[ 1 - \mleft( 1 - F_R\mleft( \frac{2}{\sqrt{(1+\alpha) \E[\Theta^\sfr]}} \mright) \mright)^2 \mright].
    \end{equation*}
    Furthermore, taking only the first term of the series for $F_R(x)$, setting $\alpha = 0.09$, and using bounds from Theorem~\ref{thm:up_low_exp_inv_inr}, we obtain
    \begin{equation*}
        \Var(\Theta^\sfr) \ge 5.2 \cdot 10^{-11}.
    \end{equation*}
\end{lemma}
\begin{proof}
    From~\eqref{eq:cheb_with_alpha} we clearly have
\begin{equation}\label{eq:var-theta-step1}
    \Var(\Theta^{\sfr}) \ge \alpha^2 (\E[\Theta^{\sfr}])^2 \p \mleft( \Theta^{\sfr} \ge (1 + \alpha) \E[\Theta^{\sfr}] \mright).
\end{equation}
From~\eqref{eq:Thetar_lb_step1} we know that $\Theta^{\sfr} \ge \inf\{t\ge 0 : \min\{R_1(t), R_2(t)\} > 2\}$. Hence,
\begin{align*}
    \p(\Theta^{\sfr} \ge u)
    & \ge \p(\min\{R_1(u), R_2(u)\} \le 2) = 1 - \p(\min\{R_1(u), R_2(u)\} > 2) \\
    & = 1 - \p(R_1(u) > 2)^2 = 1 - \p(\sqrt{u}R_1 > 2)^2 = 1 - \p \mleft( R_1 > \frac{2}{\sqrt{u}} \mright)^2, \quad \text{for all } u > 0,
\end{align*}
where we used that coordinate-wise range processes are independent and identically distributed, and the fact that standard Brownian scaling gives the same scaling for the corresponding range process. Therefore,
\begin{equation}\label{eq:var-theta-step2}
    \p(\Theta^{\sfr} \ge u) \ge 1 - \mleft( 1 - F_R\mleft( \frac{2}{\sqrt{u}} \mright) \mright)^2,
\end{equation}
where $F_R$ again denotes the cumulative distribution function from~\eqref{eq:cdf-range}. Setting $u=(1+\alpha)\E[\Theta^\sfr]$ in~\eqref{eq:var-theta-step2}, and plugging this into~\eqref{eq:var-theta-step1} yields
\begin{equation*}
    \Var(\Theta^{\sfr}) \ge \alpha^2 \big( \E[\Theta^\sfr] \big)^2 \mleft[ 1 - \mleft( 1 - F_R \mleft( \frac{2}{\sqrt{(1+\alpha)\E[\Theta^\sfr]}} \mright) \mright)^2 \mright].
\end{equation*}
Using the precise bounds from Theorem~\ref{thm:up_low_exp_inv_inr} we get
\begin{equation}\label{eq:var-theta-final}
    \Var(\Theta^{\sfr}) \ge \alpha^2 (2.61378)^2 \mleft[ 1 - \mleft( 1 - F_R \mleft( \frac{2}{\sqrt{18.7460(1+\alpha)}} \mright) \mright)^2 \mright].
\end{equation}

Recall that all the terms in the explicit series representation of $F_R$ are positive, and therefore truncating the series after finitely many terms yields a lower bound on the distribution function. Since the right-hand side of~\eqref{eq:var-theta-final} is monotone increasing in $F_R$, the resulting expression remains a valid lower bound for the variance. Optimizing the right-hand side of~\eqref{eq:var-theta-final} over $\alpha\in(0,1)$ and retaining only the first term of the series (the second term already being smaller than $10^{-90}$ at the optimizer $\alpha \approx 0.09$) gives the explicit bound
\begin{align*}
        \Var(\Theta^{\sfr})
    & \ge (0.09)^2 \cdot (2.61378)^2 \mleft[ 1 - \mleft( 1 - 8 \mleft( \frac{18.7460\cdot 1.09}{4} + \frac{1}{\pi^2} \mright) \exp{\mleft( -\frac{\pi^2\cdot 18.7460\cdot 1.09}{8}\mright)} \mright)^2 \mright] \\
    & \ge 5.2 \cdot 10^{-11}. 
\end{align*}
\end{proof}

Similar to Section~\ref{sec:up_bound_inrad_constr}, we will use a two-stage construction to find an upper bound for the variance of the inverse inradius. To prove such an upper bound, we consider the second moment objective function given by
\begin{equation}\label{eq:defini_object_func_var_inra}
        \Psi_2(m,b,a) \coloneqq  16a^4 M_\Theta(2)
+
8a^2 M_\Theta(1)\frac{m^2 b^2+a^2}{m^2-1}  +
\frac{(m^2 b^2+a^2)(5m^2-1)}
{3(m^4-6m^2+1)(m^2-1)}
\mleft(
\frac{m^2(m^2-5)}{m^2-1}b^2
+
a^2
\mright),
\end{equation} over the feasible region $\mathfrak{F}(2)$, where $M_\Theta(1)\coloneqq \E[\min\{\Theta_1(1),\Theta_2(1)\}]$ and $M_\Theta(2)\coloneqq \E[\min\{\Theta_1(1),\Theta_2(1)\}^2]$.

\begin{lemma}\label{lem:ub_var_inv_inr}
Consider the objective function $\Psi_2$ from~\eqref{eq:defini_object_func_var_inra} over the feasible region $\mathfrak{F}(2)$. Then, it follows that
\[
\E[(\Theta^\sfr)^2] \le \inf_{(m,b,a) \in \mathfrak{F}(2)}\Psi_2(m,b,a)\le 716.430.
\] Specifically, it follows that 
\begin{equation}\label{eq:var_bound_theta_sfr}
    \Var(\Theta^\sfr) \le \inf_{(m,b,a) \in \mathfrak{F}(2)}\Psi_2(m,b,a) - \mleft(2 + \frac{8}{\pi} \int_0^{\infty} \mleft( 1 - \frac{4}{(\cos t + \cosh t)^2} \mright) \frac{1}{t^3} \D t\mright)^2 \le 709.599.
\end{equation}
\end{lemma}

\begin{proof}
Recall the definition of $\tau_{m,b,a}$ from Proposition~\ref{prop:two_stage_inradius_upper_bound}, which also implies, for all $(m,b,a) \in \mathfrak{F}(2)$, that
\begin{equation}
    \E[(\Theta^\sfr)^2]\le \E\mleft[ \mleft( 4a^2 \min\{\Theta_1,\Theta_2\}+\tau_{m,b,a} \mright)^2\mright]
    = 16a^4 M_\Theta(2)+\E\mleft[\tau_{m,b,a}^2\mright]+8a^2M_\Theta(1)\E\mleft[\tau_{m,b,a}\mright],
\end{equation} where,  $\E[\tau_{m,b,a}]=(m^2b^2+a^2)/(m^2-1)$ follows from Lemma~\ref{lem:torsion}, and Lemma~\ref{lem:second_moment} moreover implies that
\begin{equation*}
\E\mleft[\tau_{m,b,a}^2\mright] =\frac{\mleft(\frac{m^2 b^2+a^2}{m^2-1}\mright)}{3(m^4-6m^2+1)}\mleft(\frac{m^2(m^2-5)}{m^2-1}(5m^2-1)b^2+(5m^2-1)a^2\mright).
\end{equation*} Hence, it follows directly, that $\E[(\Theta^\sfr)^2] \le \inf_{(m,b,a) \in \mathfrak{F}(2)}\Psi_2(m,b,a)$.

Now, we use $M_\Theta(1)\le 0.346555$ and $M_\Theta(2)\le 0.149321$. The approximation of $M_\Theta(2)$ follows from Lemma~\ref{lem:est_sec_mom_invers_range} below, while the approximation of $M_\Theta(1)$ follows from~\eqref{eq:inv_range_first_mom} (see also~\cite[Lemma 1.7]{CPS}). Taking $m = 2.706494$, $a = 3.330452$ and $b = 2.198179$ together with the bounds on $M_{\Theta}(i)$ gives the desired bound. It is straightforward to check that this point is indeed in the feasibility region $\mathfrak{F}(2)$, by checking the value of $Q_{m, b, a}(x)$ at $x = b$, the two critical points in $(b, \infty)$, and as $x \to \infty$.

Finally, applying this upper bound for $\E[(\Theta^\sfr)^2]$ together with the lower bound of $\E[\Theta^\sfr]$ in Lemma~\ref{lem:lb_exp_inv_inr}, we conclude the bound in~\eqref{eq:var_bound_theta_sfr}.
\end{proof}

\begin{lemma}\label{lem:est_sec_mom_invers_range}
Let \(W_i=(W_i(t):t\ge  0)\) be a standard one-dimensional Brownian motion, let $R_i$ be its range, and $\Theta_i$ be the inverse range. Then,
\[
\E\!\mleft[\min\{\Theta_1,\Theta_2\}^2\mright]
=
\frac13-\mathcal C, \quad \text{where}\quad \mathcal C
\coloneqq 
\frac{16}{\pi^6}
\sum_{\substack{a\ge  1\\ a\ {\rm odd}}}
\mleft[
\frac{20u_a}{a^6}
+
\frac{12\pi v_a u_a}{a^5}
+
\frac{\pi^2(2u_a-3u_a^2)}{a^4}
\mright],
\] with $u_a=\sech^2(\pi a/2)$, and $v_a=\tanh(\pi a/2)$. Moreover, every summand in the series defining \(\mathcal C\) is positive, yielding for any $N \in \N$, that 
\begin{equation*}
    \E\!\mleft[\min\{\Theta_1,\Theta_2\}^2\mright]
\le 
\frac13-\mathcal C_N , \quad \text{for}\quad \mathcal C_N
\coloneqq 
\frac{16}{\pi^6}
\sum_{\substack{1\le  a\le  N\\ a\ {\rm odd}}}
\mleft[
\frac{20u_a}{a^6}
+
\frac{12\pi v_a u_a}{a^5}
+
\frac{\pi^2(2u_a-3u_a^2)}{a^4}
\mright].
\end{equation*} Numerically, $\E\!\mleft[\min\{\Theta_1,\Theta_2\}^2\mright]
=
0.1493203187\ldots 
$.
\end{lemma}

\begin{proof}
We use the known density of \(\Theta_i\) (see~\cite[Eq.\ (1.8) \& Tab.\ 1]{theta_laws} and~\cite[Eqs.\ (20)--(22)]{CPS}), namely
\[
f_{\Theta_i}(t)
=
4\sum_{n=0}^{\infty}
\mleft((2n+1)^2\pi^2t-1\mright)
\exp\mleft(-\frac{(2n+1)^2\pi^2}{2}t\mright),
\qquad \text{for all }t>0.
\] 
Integrating term by term gives
\begin{equation}{\label{eq:survival_expansion}}
    \p(\Theta_i>t)
=
4\sum_{n=0}^{\infty}
\mleft(2t+\lambda_n^{-1}\mright)e^{-\lambda_n t}, \quad \text{ for }\lambda_n=\frac{\pi^2(2n+1)^2}{2}.
\end{equation}

Let $Y=\min\{\Theta_1,\Theta_2\}$, and since \(\Theta_1\) and \(\Theta_2\) are independent, it follows that $\p(Y>t)=\p(\Theta_1>t)^2$. For every nonnegative random variable \(X\), $\E[X^2]
=
2\int_0^\infty t \p(X>t)\D t$.
Therefore, $\E[Y^2] = 2\int_0^\infty t \p(X>t)^2\D t$. Substituting the survival expansion~\eqref{eq:survival_expansion} and applying Tonelli's theorem, gives
\[
\E[\min\{\Theta_1,\Theta_2\}^2]
=
32
\sum_{m,n=0}^\infty 
\int_0^\infty
t
\mleft(2t+\lambda_m^{-1}\mright)
\mleft(2t+\lambda_n^{-1}\mright)
e^{-(\lambda_m+\lambda_n)t}
\D t .
\]
Using the Gamma integral $\int_0^\infty t^k e^{-\alpha t}\D t
= k!/\alpha^{k+1}$, for $k=1,2,3$ and $\alpha>0$, we obtain
\begin{equation}\label{eq:bound_lambda_sum_Y_expansion}
    \E[\min\{\Theta_1,\Theta_2\}^2]
=
\sum_{m,n=0}^\infty
\mleft[
\frac{768}{(\lambda_m+\lambda_n)^4}
+
\frac{128(\lambda_m^{-1}+\lambda_n^{-1})}{(\lambda_m+\lambda_n)^3}
+
\frac{32}{\lambda_m\lambda_n(\lambda_m+\lambda_n)^2}
\mright].
\end{equation}

Now write $a=2m+1$ and $b=2n+1$. Thus, \(a,b\) range over the positive odd integers and let $\lambda_m=\pi^2a^2/2$, $\lambda_n=\pi^2b^2/2$. A direct substitution into~\eqref{eq:bound_lambda_sum_Y_expansion}, gives
\begin{equation}\label{eq:bound_lambda_sum_Y_expansion_2}
    \E[\min\{\Theta_1,\Theta_2\}^2]
=
\frac1{\pi^8}
\sum_{\substack{a,b\ge  1\\ a,b\ {\rm odd}}}
\mleft[
\frac{12288}{(a^2+b^2)^4}
+
\frac{2048(a^{-2}+b^{-2})}{(a^2+b^2)^3}
+
\frac{512}{a^2b^2(a^2+b^2)^2}
\mright].
\end{equation}

For \(k\ge  1\), define $A_k(a)=
\sum_{b\ge  1\,:\, b\ {\rm odd}}
(a^2+b^2)^{-k}$, and note that $\sum_{b\ge  1\,:\, b\ {\rm odd}} b^{-2}
= (1-2^{-2})\zeta(2)=\pi^2/8$. For fixed \(a\), the inner sum of~\eqref{eq:bound_lambda_sum_Y_expansion_2} over \(b\), may be written as
\[
12288A_4(a)
+
\frac{2048}{a^2}A_3(a)
+
2048
\sum_{\substack{b\ge  1\\ b\ {\rm odd}}}
\frac1{b^2(a^2+b^2)^3}
+
\frac{512}{a^2}
\sum_{\substack{b\ge  1\\ b\ {\rm odd}}}
\frac1{b^2(a^2+b^2)^2}.
\] Next, using the identity
\[
\frac1{b^2(a^2+b^2)^k}
=
\frac1{a^2}
\mleft[
\frac1{b^2(a^2+b^2)^{k-1}}
-
\frac1{(a^2+b^2)^k}
\mright], \quad \text{ for }k=2,3, \text{ and for all }b\ge 1 \text{ odd}, 
\]
implies that 
\[
\sum_{\substack{b\ge  1\\ b\ {\rm odd}}}
\frac1{b^2(a^2+b^2)^2}
=
\frac{\pi^2/8-A_1(a)}{a^4}
-
\frac{A_2(a)}{a^2}, \quad \text{and}\quad 
\sum_{\substack{b\ge  1\\ b\ {\rm odd}}}
\frac1{b^2(a^2+b^2)^3}
=
\frac{\pi^2/8-A_1(a)}{a^6}
-
\frac{A_2(a)}{a^4}
-
\frac{A_3(a)}{a^2}.
\]
Substitution cancels the \(A_3(a)\)-terms, and~\eqref{eq:bound_lambda_sum_Y_expansion_2} reduces to
\begin{equation}\label{eq:summanda_theta_exp_a_k}
    \E[\min\{\Theta_1,\Theta_2\}^2]
=
\frac1{\pi^8}
\sum_{\substack{a\ge  1\\ a\ {\rm odd}}}
\mleft[
\frac{320\pi^2}{a^6}
-
\frac{2560A_1(a)}{a^6}
-
\frac{2560A_2(a)}{a^4}
+
12288A_4(a)
\mright].
\end{equation}

We will now evaluate \(A_k(a)\). The Mittag--Leffler expansion $\sum_{j\in\mathbb Z}((j+1/2)^2+x^2)^{-1} =
\pi\tanh(\pi x)/x$ (follows from~\cite[Eq.~4.22.3]{NIST:DLMF} applied at $z=1/2+ix$ and then taking the imaginary part), implies, after taking \(x=a/2\), that $A_1(a)= \pi\tanh(\pi a/2)/(4a)$. Moreover, differentiating term by term gives the recurrence $A_{k+1}(a)
=
-(2ka)^{-1}\tfrac{d}{da}A_k(a)$. Put $v_a=\tanh(\pi a/2)$, and $u_a=\sech^2(\pi a/2)$, then $\tfrac{d}{da}v_a =\pi u_a/2$ and $\tfrac{d }{da} u_a= -\pi u_av_a$. Using the recurrence for $A_k$, it follows that 
\[
A_2(a)
=
\frac{\pi}{16a^3}
\mleft(2v_a-\pi a u_a\mright), \quad \text{and}\quad A_4(a)
=
\frac{\pi}{768a^7}
\mleft(
60v_a
-
30\pi a u_a
-
12\pi^2a^2u_av_a
-
2\pi^3a^3u_a
+
3\pi^3a^3u_a^2
\mright).
\]
Substituting these expressions into the summands from~\eqref{eq:summanda_theta_exp_a_k}, yields the simplification
\[
\frac{320\pi^2}{a^6}
-
\frac{2560A_1(a)}{a^6}
-
\frac{2560A_2(a)}{a^4}
+
12288A_4(a)
=
\frac{320\pi^2}{a^6}
-
16\pi^2
\mleft[
\frac{20u_a}{a^6}
+
\frac{12\pi v_a u_a}{a^5}
+
\frac{\pi^2(2u_a-3u_a^2)}{a^4}
\mright].
\]
Hence, by~\eqref{eq:summanda_theta_exp_a_k}, it follows that
\[
\E[\min\{\Theta_1,\Theta_2\}^2]
=
\frac{320}{\pi^6}
\sum_{\substack{a\ge  1\\ a\ {\rm odd}}}\frac1{a^6}
-
\frac{16}{\pi^6}
\sum_{\substack{a\ge  1\\ a\ {\rm odd}}}
\mleft[
\frac{20u_a}{a^6}
+
\frac{12\pi v_a u_a}{a^5}
+
\frac{\pi^2(2u_a-3u_a^2)}{a^4}
\mright]=\sum_{\substack{a\ge  1\\ a\ {\rm odd}}}\frac1{a^6}
-\mathcal C.
\]
By definition of the Riemann zeta function~\eqref{eq_riemann_zeta_function_def}, it holds that
\[
\frac{320}{\pi^6} \sum_{\substack{a\ge  1\\ a\ {\rm odd}}}\frac1{a^6} = \frac{320}{\pi^6} \mleft(\zeta(6)-2^{-6}\sum_{n=1}^\infty \frac{1}{n^6}\mright)
=
\frac{320}{\pi^6}\mleft(1-2^{-6}\mright)\zeta(6)
=
\frac{320}{\pi^6} \mleft(\frac{63}{64}\cdot \frac{\pi^6}{945}\mright)=\frac{1}{3}.
\]
This proves $\E[\min\{\Theta_1,\Theta_2\}^2]
=
1/3-\mathcal C$, which was the first result of the statement. 

It remains to prove that every summand in the definition of \(\mathcal C\) is
positive. Fix an odd integer \(a\ge  1\). Recall that $v_a=\tanh(\pi a/2)$, and $u_a=\sech^2(\pi a/2)$, for which it holds that $0<u_a\le  \sech^2(\pi/2)$ and $0 < v_a <1$, since \(a\ge  1\). Moreover, $\sech^2(\pi/2)<2/3$, implying directly, that $0<u_a<2/3$. Consequently $2u_a-3u_a^2=u_a(2-3u_a)>0$. Hence, each of the three terms $20u_a/a^6$, $12\pi v_a u_a/a^5$, and $\pi^2(2u_a-3u_a^2)/a^4$ is strictly positive, implying that each summand of $\mathcal C$ is strictly positive as claimed. Since each summand is strictly positive, and the series $\mathcal C$ is absolutely convergent, it follows directly for all $N \in \N$, that $\mathcal C_N\le  \mathcal C$, and hence $\E\!\mleft[\min\{\Theta_1,\Theta_2\}^2\mright]
\le 1/3-\mathcal C \le 
1/3-\mathcal C_N$.
\end{proof}

\begin{proof}[Proof of Theorem~\ref{thm:up_low_var_inv_inr}]
    The proof of the lower bound follows directly from Lemma~\ref{lem:lb_var_inv_inr}, and the proof of the upper bound follows directly from Lemma~\ref{lem:ub_var_inv_inr}.
\end{proof}

\section*{Acknowledgements}
We thank Mateusz Kwa\'{s}nicki for suggesting the use of a hyperbola-shaped domain to compute the upper bound for $\E[\Theta^\sfr]$ that appears Theorem~\ref{thm:up_low_exp_inv_inr}.

DKB is supported by AUFF NOVA grant AUFF-E-2022-9-39. S\v{S} is supported by the \emph{Croatian Science Foundation} grant IP-2022-10-2277. This research was also funded by the European union-NextGenerationEU through the National Recovery and Resilience Plan 2021-2026 Institutional grant of University of Zagreb Faculty of Electrical Engineering and Computing (VALOR). This work was carried out within a project DIGIT.2.1.02.016 funded by the Digital, Innovation, and Green Technology Project – DIGIT Project (IBRD Loan No. 9558‑HR).

\printbibliography

\end{document}